\documentclass[11pt]{amsart}

\usepackage[T1]{fontenc}
\usepackage[utf8]{inputenc}
\usepackage{textcomp}
\usepackage{lmodern}
\usepackage{microtype}
\usepackage{tikz}

\usepackage{amssymb}
\usepackage{amsthm}
\usepackage{amscd}

\usepackage{hyperref}

\newtheorem{theorem}{Theorem}[section]

\newtheorem{lemma}[theorem]{Lemma}
\newtheorem{prop}[theorem]{Proposition}
\newtheorem{cor}[theorem]{Corollary}

\theoremstyle{definition}
\newtheorem{definition}[theorem]{Definition}

\newtheorem{remark}[theorem]{Remark}

\newtheorem{claim}{Claim}

 \newenvironment{claimproof}{\begin{proof}}{\end{proof}}

 \def \GS{\operatorname{GS}}
  \def \NFOP{\operatorname{NFOP}}
    
      \def \VC{\operatorname{VC}}
    \def \NXOP{\operatorname{NXOP}}
  \def\disc{\operatorname{disc}}

\allowdisplaybreaks[2]

\begin{document}

\title{Averages of hypergraphs and higher arity stability}
\author{Artem Chernikov and Henry Towsner}
\date{\today}

\begin{abstract}
We show that $k$-ary functions giving the measure of the intersection of multi-parametric families of sets in probability spaces, e.g.~$(x,y,z) \in X \times Y \times Z \mapsto \mu(P_{x,y} \cap Q_{x,z} \cap R_{y,z})$, satisfy a particularly strong form of hypergraph regularity. More generally, this applies to the (integral) averages of continuous combinations of functions of smaller arity. This result is connected to higher arity stability in model theory, that we discuss in the second part of the paper. We demonstrate that all hypergraphs embedding both into the half-simplex and into $\GS(\mathbb{F}_3)$, the two known sources of failure of ternary stability, do satisfy an analogous regularity lemma --- hence strong ternary stability cannot be characterized simply by excluded hypergraphs.
\end{abstract}

\maketitle

\section{Introduction}

Consider a function $f: X \times Y \to \mathbb{R}_{\geq 0}$ whose value is the measure of an intersection of measurable sets---we have families $\{P_x\}_{x\in X}$ and $\{Q_y\}_{y\in Y}$ in some probability space and set
\[f(x,y) := \mu(P_x\cap Q_y).\]
Such functions turn out to be quite special. Combining model-theoretic stability of such functions (\cite{hrushovski2012stable}, see discussion below) with an improved form of Szemer\'edi's regularity lemma for stable graphs \cite{malliaris2014regularity} (or more precisely, its continuous version, see e.g.~\cite{chavarria2024continuous}) shows that any function of this kind is almost constant on a small number of rectangles: for any $\varepsilon>0$ there are partitions $X=\bigcup_{i\leq n}X_i$ and $Y=\bigcup_{j\leq n}Y_j$ with $n =n (\varepsilon)$ so that for each rectangle $X_i\times Y_j$ there is a value $\alpha_{i,j} \in [0,1]$ and a small exceptional set $Z\subseteq X_i\times Y_j$ with $|Z|<\varepsilon|X_i\times Y_j|$ so that $f[X_i\times Y_j\setminus Z]\subseteq (\alpha_{i,j}-\varepsilon,\alpha_{i,j}+\varepsilon)$.

Our first result in this paper is a higher-arity analog of this result in which the measurable sets are indexed by multiple overlapping coordinates. Specifically, fix a $k$, we suppose we have sets $X_1,\ldots,X_k$ and, for each $e\in\binom{[k]}{k-1}$ (or, more generally, $e\in\binom{[k]}{d}$ for any $d<k$; where $[k] = \{ 1, 2, \ldots, k \}$), a family of measurable sets $\{P^e_{\vec x}\}_{\vec x\in\prod_{i\in e}X_i}$. We can define a function $f:\prod_{i \in [k]}X_i\rightarrow[0,1]$ by
\[f(\vec x) := \mu \left(\bigcap_{e \in\binom{[k]}{k-1} } P^e_{\vec x_e} \right),\]
where the notation $\vec x_e$ means the sub-tuple of $\vec x$ with coordinates from $e$---when $\vec x$ is $(x_1,\ldots,x_k)$, $\vec x_e$ is $(x_i)_{i\in e}$. We show:
\begin{quote}
for any $\varepsilon>0$ there exist partitions $\prod_{i\in e}X_i=\bigcup_{j\leq n}X^e_j$ with $n = n(\varepsilon, k)$ for each $e\in\binom{[k]}{k-1}$ so that each $|X^e_0|<\varepsilon|\prod_{i\in e}X_i|$ and for each \emph{cylinder intersection set} $\bigcap_e X^e_{j_e} = \{\vec{x} \in \prod_{i\in [k]}X_i : \bigwedge_{e} (\vec{x}_{e} \in X^e_{j_e}) \}$ with every  $j_e>0$, there is a value $\alpha_{\{j_e\}} \in [0,1]$ and an exceptional set $Z\subseteq \bigcap_eX^e_{j_e}$ with $|Z|<\varepsilon|\bigcap_e X^e_{j_e}|$ so that
\[f \left[ \bigcap_eX^e_{j_e}\setminus Z \right]\subseteq \left(\alpha_{\{j_e\}}-\varepsilon, \alpha_{\{j_e\}}+\varepsilon \right).\]
\label{state_main}
\end{quote}
A stronger version of this result is Theorem \ref{thm:main}. Dually, we can think about this as a presentation theorem for the averages of hypergraphs (which explains the title). Namely, for every $z \in \Omega$ we are given a $k$-hypergraph $E_z \subseteq \prod_{i \in [k]} X_i$ defined by the cylinder intersection of $d$-ary hypergraphs $E_z = \bigcap_{e} P^{e}_z$ with $P^{e}_z = \{\vec{x} \in \prod_{i \in e} X_i : z \in P^e_{\vec{x}}\}$, and the $[0,1]$-function  $f(\vec{x}) = \mu \left( \bigcap_{e} P^e_{\vec x_e}  \right) = \int_{\Omega} E_z d\mu(z)$ is the ``average'' of the hypergraphs $E_z$ over all $z \in \Omega$. Our result shows that such ``average'' hypergraphs satisfy a particularly strong form of hypergraph regularity. In fact, in Corollary \ref{cor: part for integrals} we show that the same conclusion holds for more general functions of the form $f(\vec x) := \int_{\Omega} h \left( \left( f^e_{\vec x_e}(z) : e\in\binom{[k]}{d} \right)  \right) d \mu(z)$ where $h: [0,1]^{\binom{k}{d} } \to [0,1]$ is an arbitrary continuous function and   for each $e\in\binom{[k]}{d}$ and each $\vec x\in\prod_{i\in e}X_i$, $f^e_{\vec x}: \Omega \to [0,1]$ is a measurable function (e.g.~functions of the form $f(x_1,x_2,x_3) = \int_{G} g(x_1 \cdot x_2 \cdot z) \cdot  h(x_1 \cdot x_3 \cdot z) \cdot \iota (x_2 \cdot x_3 \cdot z)   dz$, for arbitrary  functions $g, h, \iota: G \to [0,1]$ on an amenable group $G$). 

When $k=2$ and $d=1$, the result in the previous paragraph may appear slightly weaker than the decomposition  guaranteed by stable graph regularity for $\mu(P_x\cap Q_y)$ which we described above: instead of \emph{every} rectangle having the desired property, what this implies is that there are small exceptional sets $X^1_0$ and $X^2_0$, and only the rectangles avoiding them have the desired property. It is easy to recover the version with no exceptional rectangles, however, because we can ``hide'' the sets $X^1_0$ and $X^2_0$ inside the other sets: we partition $X^1_0$ into small pieces $X^1_0=\bigcup_i X^*_i$ where the size of $X^*_i$ is proportionate to the size of $X^1_i$, and then use the new partition $X_1=\bigcup_i (X^1_i\cup X^*_i)$. Even with the exceptional rectangles, this is still stronger than what graph regularity would tell us about an arbitrary function: graph regularity says that there are a small number of exceptional rectangles, but this result implies that these exceptional rectangles are determined by a single coordinate.

The version with no exceptional rectangles at all does not appear to generalize to higher arity. Instead, we prove a different strengthening of: instead of the exceptional sets $Z$ having a fixed density $\varepsilon$, we can fix a function $F: \mathbb{N} \to (0,1]$ and ask that $Z$ have density $<F(m)$ where $m$ is the size of the partitions (see Theorem \ref{thm:main} for the exact statement). That is, we can ask for an analog of strong graph regularity \cite{alon2007efficient} instead of ordinary graph regularity. Even when $k=2,d=1$, this strengthening requires us to allow exceptional rectangles in the restricted way above. See \cite{chernikov2024perfect} for a more detailed discussion of these different versions in the $k=3$, $d=1$ setting.

While not necessary to understand the statement or proof, the motivation for this result comes from the investigation of tame regularity lemmas and connections to model theory (we refer to the introduction of \cite{chernikov2024perfect} for a brief survey of the subject). Szemer\'edi's regularity lemma guarantees that any function $f$ admits a partition into rectangles so that, on most rectangles, $f$ is quasi-random. The result of \cite{malliaris2014regularity}, often called ``stable regularity'', says that if $f$ is \emph{stable} then this partition can be improved so that $f$ is nearly constant (not just quasi-random) on every (not just most) rectangle. Stability is mostly studied in model theory, but has a simple combinatorial characterization in terms of excluded configurations.

\begin{definition}\label{def:stable}
 A function $f: X \times Y \to [0,1]$ is stable if, for each $\delta > 0$, there is a $k \in \mathbb{N}$ so that there do not exist sequences $x_1,\ldots,x_k$ in $X$ and $y_1,\ldots,y_k$ in $Y$ and a value $\alpha \in [0,1]$ so that:
  \begin{itemize}
  \item whenever $i<j$, $f(x_i,y_j)<\alpha$,
  \item whenever $j<i$, $f(x_i,y_j)>\alpha+\delta$.
  \end{itemize}
\end{definition}

\noindent It appears explicitly in \cite{hrushovski2012stable} that when $f(x,y)=\mu(P_x\cap Q_y)$, $f$ is stable in this sense. 
The result follows from stability of probability algebras in continuous logic \cite{yaacov2013theories, yaacov2008model}, and can be traced back to \cite{krivine1981espaces} and even \cite{grothendieck1952criteres}. 
See also \cite{Tao} for an elementary proof. This result found various applications, e.g.~in the classification of approximate subgroups \cite{hrushovski2012stable}, or a simplified proof of algebraic regularity lemma \cite{zbMATH06476712, arXiv:1310.7538, Tao}.

When $f$ is a function with $k$ inputs for $k>2$, regularity gets replaced by hypergraph regularity \cite{nagle2006counting, rodl2004regularity, gowers2007hypergraph, tao2006variant}. Recent work \cite{chernikov2020hypergraph, terry2021higher, terry2021irregular, chernikov2024perfect} has attempted to identify higher-arity analogs of model theoretic tameness notions by hypothesizing what the analogous version of hypergraph regularity would be.

Terry and Wolf showed \cite{terry2021irregular, terry2021higher} that there are at least two distinct potential notions of ``ternary stability'' (i.e.~stability for ternary relations; also called \emph{$2$-stability} as a property of ternary relations approximable by binary relations in an appropriate sense). They call the weaker of these notions $\NFOP_2$ (``not the functional order property''), but the stronger notion is not well understood and does not yet have a name (they refer to it as $\NXOP_2$, where $X$ indicates that the property is unknown). For our purposes here, we will call it ``strong $2$-stability''. (We expect that a real name awaits a clearer picture of the various notions of higher-arity stability.) Our result here suggests that the function $f(x,y,z)=\mu(P^{1,2}_{x,y}\cap P^{1,3}_{x,z}\cap P^{2,3}_{y,z})$ has strong $2$-stability, along with analogous results for functions of $k+1 \geq 3$  variables and strong $k$-stability.\footnote{A fuller account of strong $2$-stability would include both a combinatorial formulation and an equivalent regularity-type property. The property considered here seems to be the most natural finitistic consequence of the correct ``perfect'' regularity-type property that holds in the ultraproduct (see the discussion in Section \ref{sec: ultraprod and perf reg}), but may not actually be equivalent. See \cite{chernikov2024perfect} for a detailed investigation of this issue---the distinction between perfect, strong and ordinary regularity---for ordinary stability.} 

The relevant result from \cite{terry2021irregular, terry2021higher} is to identify some ternary relations which do \emph{not} have strong $2$-stability. They show that the half-simplex $\{(x,y,z)\in[0,1]\mid x+y+z<1\}$ does not have strong $2$-stability, and neither do the relations $\GS_n$, $n \geq 3$, on the set of infinite sequences from $\{0,1,\ldots,n-1\}$ where $(x,y,z)\in \GS_n$ if, for the first $i$ for which $x(i)+y(i)+z(i)\mod n\neq 0$, $x(i)+y(i)+z(i)=1$ (Definition \ref{def: half-simp and GS}). Neither of these two hypergraphs embeds into the other (as an induced sub-hypergraph, see Section \ref{sec: half-simp and GS}), therefore neither can serve as a family of excluded structures characterizing strong $2$-stability analogous to the characterization of stability given by Definition \ref{def:stable}.

In Theorem \ref{thm: embeds into GS3 and simplex}, we show that any family of $3$-hypergraphs that embed both into the half-simplex and into $\GS_3$ \emph{do} satisfy strong $2$-stability. In fact we show a stronger result: any ultraproduct of \emph{monotone} hypergraphs (Definition \ref{def: monotone}) that embed into $\GS_3$  satisfies what we call \emph{perfect} $2$-stable regularity lemma  (see Definition \ref{def: perf 2-stab reg}, and Proposition \ref{prop: perf implies strong}). It follows that, unlike stability and $\NFOP_2$ (and other related notions like $n$-dependence/finite $\VC_k$-dimension \cite{chernikov2019n, chernikov2020hypergraph}), strong $2$-stability cannot be characterized simply by excluded hypergraphs.

Hrushovski's proof that $\mu(P_x\cap Q_y)$ is stable relied on the combinatorial characterization and an Aldous-Hoover-Kallenberg presentation for exchangeable  arrays of random variables. Similarly, our earlier result that the functions considered in Theorem \ref{thm:main} (and Corollary \ref{cor: part for integrals}) have a weaker property of finite VC$_k$ dimension relied on a combinatorial characterization, structural Ramsey theory and an Aldous-Hoover-Kallenberg style presentation (this is a special case of 
\cite[Theorem 10.7]{chernikov2020hypergraph}, see also \cite{arXiv:2406.18772} for an exposition).  In the absence of a combinatorial characterization of strong $2$-stability, we cannot imitate these proofs. Instead, we give a new argument, constructing the partition directly by manipulating partitions coming from hypergraph regularity. For a connection of these type of results to higher amalgamation/uniqueness in model theory, see the discussion \cite[Appendix B]{hrushovski2024approximate} and \cite{CheOber}.

\section{Averages of hypergraphs (and of continuous combinations of functions)}\label{sec:main}

For $P \subseteq \prod_{i \in [k]} X_i$, $e \subseteq [k]$ and $\vec{x} = (x_i : i \in e) \in \prod_{i \in e} X_i$, we write $P_{\vec{x}} = \{ \vec{x}' \in \prod_{i \in [k] \setminus e} X_i: \bigwedge_{i \in e} (x_i = x'_i) \}$ for the \emph{slice of $P$ at $\vec{x}$}. For  $\vec{x} = (x_i : i \in [k]) \in \prod_{i \in [k]} X_i$, we write $\vec x_e = (x_i : i\in e)$ for the sub-tuple of $\vec x$ with coordinates from $e$ (if $e = \emptyset$, $\vec{x}_e = \langle \rangle$ is the empty tuple, and $P_{\langle \rangle} = P$).

\begin{theorem}\label{thm:main}
 
 For every $1 \leq d < k \in \mathbb{N}$, $\varepsilon \in \mathbb{R}_{>0}$ and function $F: \mathbb{N} \to (0,1]$ there exists $N = N(d,k,\varepsilon, F) \in \mathbb{N}$ satisfying the following.
 
  Let finite sets $\{X_i\}_{1 \leq i\leq k}$ and a probability measure space $(\Omega,\mathcal{B},\mu)$ be given.  For each $e\in\binom{[k]}{d}$ and each $\vec x\in\prod_{i\in e}X_i$, let $P^e_{\vec x} \in \mathcal{B}$ be a measurable subset of $\Omega$. Consider the function $f:\prod_{1 \leq i\leq k}X_i\rightarrow[0,1]$ given by 
$f(\vec x) := \mu \left(\bigcap_{e\in\binom{[k]}{d}}P^e_{\vec x_e} \right)$.

  Then there exist partitions $\prod_{i\in e}X_i=\bigsqcup_{0 \leq j \leq b_{e}}S_{e,j}$ with $b_e \leq N$ for each $e\in\binom{[k]}{d}$ so that $|S_{e,0}| < \varepsilon |\prod_{i\in e}X_i|$ and for every non-empty cylinder intersection set $C=\bigcap_{e\in\binom{[k]}{d}}S_{e,j_e}$ with each $j_e>0$, there is an interval $I\subseteq [0,1]$ with $|I|<\varepsilon$ and a set $Z\subseteq C$ with $|Z|<F(\max_e \{b_e\}) |C|$ so that $f(\vec x)\in I$ for all $\vec x \in C \setminus Z$.
\end{theorem}
\begin{proof}
  It will be convenient to write $X_{k+1}$ for $\Omega$ and to use measure-theoretic notation (i.e.~when $A\subseteq\prod_{i\in e}X_i$, we write $\mu(A)$ for $\frac{|A|}{|\prod_{i\in e}X_i|}$). 
  For each $e\in\binom{[k]}{d}$ we consider the $(d+1)$-partite $(d+1)$-uniform hypergraph $P^e \subseteq (\prod_{i \in e} X_i) \times X_{k+1}$ defined via 
  $$P^e := \left\{(x_i)_{i \in e} {}^{\frown} x_{k+1} : (x_i)_{i \in e} \in \prod_{i\leq k+1}X_i \land x_{k+1} \in P^{e}_{(x_i)_{i \in e}}\right\}.$$
 For each $e\in\binom{[k]}{d}$,  we apply hypergraph regularity with suitable parameters $\varepsilon' \ll \varepsilon$ and $F' \ll F$ to the $(d+1)$-uniform hypergraph $P^{e}$ obtaining partitions of $\prod_{i\in e'}X_i$ for each $e' \in\binom{e \cup \{k+1\}}{ \leq d}$. And for each $e' \in\binom{[k+1]}{ \leq d}$, we obtain a partition $\prod_{i\in e'}X_i=\bigsqcup_{0 \leq j\leq k_{e'}} A_{e',j}$ by intersecting all of the partitions of $\prod_{i\in e'}X_i$ that arose from  the regular partitions for $P^{e}$ for different 
 $e \in\binom{[k]}{ d}$. We routinely abuse notation to view the sets $A_{e,j}$ as subsets of $\prod_{i\leq k}X_i$ or $\prod_{i\leq k+1}X_i$ as needed.

  We will be interested in cylinder intersection sets of various kinds. The most important will be the complete cylinder intersection sets of the form $\bigcap_{e\in\binom{[k+1]}{\leq d}}A_{e,i_e}$ for some choice of the indices $0 \leq i_e \leq k_e$, which we will call \emph{cells}.  When we have an $e\in\binom{[k]}{d}$, we will call a cylinder intersection set of the form $\bigcap_{e'\in\binom{e\cup\{k+1\}}{\leq d}}A_{e',i_{e'}}$ for some choice of $i_{e'}$'s an \emph{$e$-face}, or just a \emph{face}; and we call those of the form $\bigcap_{e\in\binom{[k]}{\leq d}}A_{e,i_e}$ a \emph{base}. We will typically call cells $C$, and write $C_e$ for its $e$-face and $C_{[k]}$ for its base.

  To motivate this terminology, consider two cases.   
  
  \

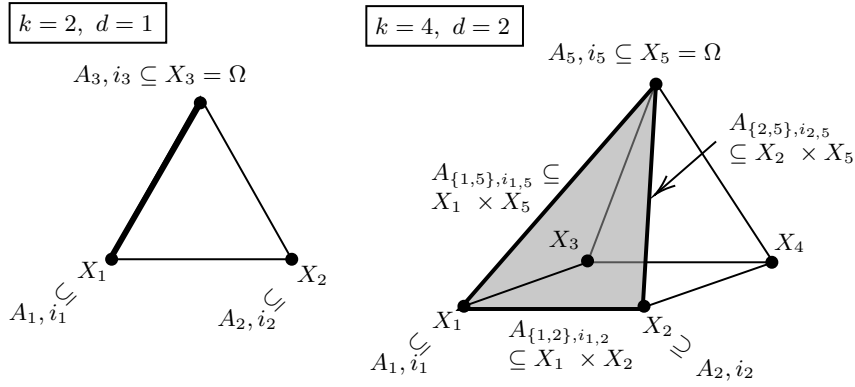
\tikzset{every picture/.style={line width=0.75pt}} 

\begin{tikzpicture}[x=0.75pt,y=0.75pt,yscale=-1,xscale=1]

\draw   (94.88,46.73) -- (140.99,128.44) -- (48.76,128.44) -- cycle ;
\draw   (324.7,39.34) -- (383.82,129.97) -- (290.17,129.97) -- cycle ;
\draw  [fill={rgb, 255:red, 155; green, 155; blue, 155 }  ,fill opacity=0.54 ][line width=1.5]  (324.7,40.04) -- (317.97,153.55) -- (225.75,153.55) -- cycle ;
\draw    (290.17,129.97) -- (225.75,152.86) ;
\draw    (383.82,129.97) -- (317.97,152.86) ;
\draw  [fill={rgb, 255:red, 0; green, 0; blue, 0 }  ,fill opacity=1 ] (92.11,49.5) .. controls (92.11,47.97) and (93.35,46.73) .. (94.88,46.73) .. controls (96.41,46.73) and (97.65,47.97) .. (97.65,49.5) .. controls (97.65,51.03) and (96.41,52.27) .. (94.88,52.27) .. controls (93.35,52.27) and (92.11,51.03) .. (92.11,49.5) -- cycle ;
\draw  [fill={rgb, 255:red, 0; green, 0; blue, 0 }  ,fill opacity=1 ] (138.22,128.44) .. controls (138.22,126.91) and (139.46,125.67) .. (140.99,125.67) .. controls (142.52,125.67) and (143.76,126.91) .. (143.76,128.44) .. controls (143.76,129.97) and (142.52,131.21) .. (140.99,131.21) .. controls (139.46,131.21) and (138.22,129.97) .. (138.22,128.44) -- cycle ;
\draw  [fill={rgb, 255:red, 0; green, 0; blue, 0 }  ,fill opacity=1 ] (47.38,128.44) .. controls (47.38,126.91) and (48.62,125.67) .. (50.15,125.67) .. controls (51.68,125.67) and (52.92,126.91) .. (52.92,128.44) .. controls (52.92,129.97) and (51.68,131.21) .. (50.15,131.21) .. controls (48.62,131.21) and (47.38,129.97) .. (47.38,128.44) -- cycle ;
\draw  [fill={rgb, 255:red, 0; green, 0; blue, 0 }  ,fill opacity=1 ] (321.93,40.04) .. controls (321.93,38.51) and (323.17,37.27) .. (324.7,37.27) .. controls (326.23,37.27) and (327.47,38.51) .. (327.47,40.04) .. controls (327.47,41.57) and (326.23,42.81) .. (324.7,42.81) .. controls (323.17,42.81) and (321.93,41.57) .. (321.93,40.04) -- cycle ;
\draw  [fill={rgb, 255:red, 0; green, 0; blue, 0 }  ,fill opacity=1 ] (315.89,152.16) .. controls (315.89,150.64) and (317.13,149.4) .. (318.66,149.4) .. controls (320.19,149.4) and (321.43,150.64) .. (321.43,152.16) .. controls (321.43,153.69) and (320.19,154.93) .. (318.66,154.93) .. controls (317.13,154.93) and (315.89,153.69) .. (315.89,152.16) -- cycle ;
\draw  [fill={rgb, 255:red, 0; green, 0; blue, 0 }  ,fill opacity=1 ] (225.06,152.16) .. controls (225.06,150.64) and (226.3,149.4) .. (227.83,149.4) .. controls (229.36,149.4) and (230.6,150.64) .. (230.6,152.16) .. controls (230.6,153.69) and (229.36,154.93) .. (227.83,154.93) .. controls (226.3,154.93) and (225.06,153.69) .. (225.06,152.16) -- cycle ;
\draw  [fill={rgb, 255:red, 0; green, 0; blue, 0 }  ,fill opacity=1 ] (287.4,129.28) .. controls (287.4,127.75) and (288.64,126.51) .. (290.17,126.51) .. controls (291.7,126.51) and (292.94,127.75) .. (292.94,129.28) .. controls (292.94,130.81) and (291.7,132.05) .. (290.17,132.05) .. controls (288.64,132.05) and (287.4,130.81) .. (287.4,129.28) -- cycle ;
\draw  [fill={rgb, 255:red, 0; green, 0; blue, 0 }  ,fill opacity=1 ] (380.36,129.97) .. controls (380.36,128.44) and (381.6,127.2) .. (383.13,127.2) .. controls (384.66,127.2) and (385.9,128.44) .. (385.9,129.97) .. controls (385.9,131.5) and (384.66,132.74) .. (383.13,132.74) .. controls (381.6,132.74) and (380.36,131.5) .. (380.36,129.97) -- cycle ;
\draw    (355,69) -- (325.51,94.69) ;
\draw [shift={(324,96)}, rotate = 318.95] [color={rgb, 255:red, 0; green, 0; blue, 0 }  ][line width=0.75]    (10.93,-3.29) .. controls (6.95,-1.4) and (3.31,-0.3) .. (0,0) .. controls (3.31,0.3) and (6.95,1.4) .. (10.93,3.29)   ;
\draw [line width=2.25]    (94.88,49.5) -- (49.38,128.44) ;

\draw (31.25,128.43) node [anchor=north west][inner sep=0.75pt]  [font=\footnotesize]  {$X_{1}$};
\draw (142.13,130.45) node [anchor=north west][inner sep=0.75pt]  [font=\footnotesize]  {$X_{2}$};
\draw (29.01,29.02) node [anchor=north west][inner sep=0.75pt]  [font=\footnotesize]  {$A_{3} ,i_{3} \subseteq X_{3} =\Omega $};
\draw    (-1.31,-0.85) -- (77.69,-0.85) -- (77.69,19.15) -- (-1.31,19.15) -- cycle  ;
\draw (1.69,3.55) node [anchor=north west][inner sep=0.75pt]  [font=\footnotesize]  {$k=2,\ d=1$};
\draw (210.31,153.31) node [anchor=north west][inner sep=0.75pt]  [font=\footnotesize]  {$X_{1}$};
\draw (316.25,156.08) node [anchor=north west][inner sep=0.75pt]  [font=\footnotesize]  {$X_{2}$};
\draw (268.58,18.31) node [anchor=north west][inner sep=0.75pt]  [font=\footnotesize]  {$A_{5} ,i_{5} \subseteq X_{5} =\Omega $};
\draw (269.16,112.45) node [anchor=north west][inner sep=0.75pt]  [font=\footnotesize]  {$X_{3}$};
\draw (382.73,113.84) node [anchor=north west][inner sep=0.75pt]  [font=\footnotesize]  {$X_{4}$};
\draw    (178.61,0.68) -- (257.61,0.68) -- (257.61,20.68) -- (178.61,20.68) -- cycle  ;
\draw (181.61,5.08) node [anchor=north west][inner sep=0.75pt]  [font=\footnotesize]  {$k=4,\ d=2$};
\draw (-3.2,149.55) node [anchor=north west][inner sep=0.75pt]  [font=\footnotesize]  {$A_{1} ,i_{1}$};
\draw (17.72,145.88) node [anchor=north west][inner sep=0.75pt]  [font=\footnotesize,rotate=-316.39]  {$\subseteq $};
\draw (203.19,76.67) node [anchor=north west][inner sep=0.75pt]  [font=\footnotesize]  {$ \begin{array}{l}
A_{\{1,5\} ,i_{1,5}} \subseteq \\
X_{1} \ \times X_{5}
\end{array}$};
\draw (353.3,51.67) node [anchor=north west][inner sep=0.75pt]  [font=\footnotesize]  {$ \begin{array}{l}
A_{\{2,5\} ,i_{2,5}}\\
\subseteq X_{2} \ \times X_{5}
\end{array}$};
\draw (241.89,155.69) node [anchor=north west][inner sep=0.75pt]  [font=\footnotesize]  {$ \begin{array}{l}
A_{\{1,2\} ,i_{1,2}}\\
\subseteq X_{1} \ \times X_{2}
\end{array}$};
\draw (178.88,174.65) node [anchor=north west][inner sep=0.75pt]  [font=\footnotesize]  {$A_{1} ,i_{1}$};
\draw (196.16,171.06) node [anchor=north west][inner sep=0.75pt]  [font=\footnotesize,rotate=-316.39]  {$\subseteq $};
\draw (343.07,177.54) node [anchor=north west][inner sep=0.75pt]  [font=\footnotesize]  {$A_{2} ,i_{2}$};
\draw (336.18,163.67) node [anchor=north west][inner sep=0.75pt]  [font=\footnotesize,rotate=-35.76]  {$\supseteq $};
\draw (102,152.55) node [anchor=north west][inner sep=0.75pt]  [font=\footnotesize]  {$A_{2} ,i_{2}$};
\draw (122.92,148.88) node [anchor=north west][inner sep=0.75pt]  [font=\footnotesize,rotate=-316.39]  {$\subseteq $};
\end{tikzpicture}

When $k=2,d=1$, we should think of a cell as a triangle pointing up, with $X_1,X_2$ on the bottom and $X_3=\Omega$ as the top point. The faces are the vertical sides of this triangle, of the form $A_{1,i_1}\times A_{3,i_3}$ or $A_{2,i_2}\times A_{3,i_3}$, and the bottom $A_{1,i_1}\times A_{2,i_2}$ is the base.  When $k=4, d=2$, a cell is a square pyramid, with $X_1,\ldots,X_4$ forming the corners of the base and $X_5=\Omega$ as the tip; the faces are then the triangluar sides of the pyramid, like $A_{1,i_1}\cap A_{2,i_2}\cap A_{5,i_5}\cap A_{\{1,2\},i_{\{1,2\}}}\cap A_{\{1,5\},i_{\{1,5\}}}\cap A_{\{2,5\},i_{\{2,5\}}}$, while the base is precisely the square base of the pyramid.

  Finally, for $e\in\binom{[k]}{d}$, we will need to consider cylinder intersections of the form $\bigcap_{e'\in\binom{e}{\leq d}}A_{e',i_{e'}}$---that is, ``\emph{sides}'', the common part of a face and the base.  In particular, we will identify a small number of bad sides to be our bad component $S_{e,0}$, and then take $S_{e,j}$ for $j>0$ to enumerate all other sides.

Fix $e \in \binom{[k]}{d}$. When $C_e = \bigcap_{e'\in\binom{e\cup\{k+1\}}{\leq d}}A_{e',i_{e'}}$ is an $e$-face,  $P^e$ has a density $\gamma_{C_e}:= \frac{\mu(P^e\cap C_e)}{\mu(C_e)}$ on $C_e$. The choice of the partitions by hypergraph regularity ensures that the measure of the union of all $e$-faces $C_e \subseteq X_e \times X_{k+1}$ for which $P^e$ sits (sufficiently)  quasi-randomly in $C_e$ is $\geq 1 - \varepsilon_1$ (by which we mean that $P^e$ is quasi-random in $C_e$ \emph{and} $A_{e',i_{e'}}$ is quasi-random in $\bigcap_{e'' \subsetneq e'} A_{e'',i_{e''}}$ for all $e'\in \binom{e\cup\{k+1\}}{\leq d}$).  We let $S_{e,0}$ be the union of those sides $C_{e,[k]} = \bigcap_{e'\in\binom{e}{\leq d}}A_{e',i_{e'}}$ for which  there are many $e$-faces $C_e \subseteq C_{e,[k]}\times X_{k+1}$ so that $P^e$ is not sufficiently quasi-random in $C_e$ (i.e., the measure of the union of all such faces $C_e$ is close to $ \mu(C_{e,[k]}\times X_{k+1}) = \mu(C_{e,[k]})$). As no face can be contained simultaneously in two different sets $C_{e,[k]}\times X_{k+1}$, we must have $\mu(S_{e,0})$ very close to $0$. As stated in the previous paragraph, we then let the rest of the partition $S_{e,j}$ enumerate all other sides.

Now fix $j_e > 0$ for each $e \in \binom{[k]}{d}$, and consider the cylinder intersection  $C_{[k]} = \bigcap_{e\in\binom{[k]}{d}}S_{e,j_e}$. We need to show that the values of  $ \mu(\bigcap_e P^e_{\vec x_e})$ are close for most $\vec x\in C_{[k]}$.

Note that $C_{[k]}$ is precisely a base $C_{[k]}=\bigcap_{e\in\binom{[k]}{\leq d}}A_{e,i_e}$ for some values $\{i_e\}$ with the property that for each side $C_{e,[k]}=\bigcap_{e'\in\binom{e}{\leq d}}A_{e',i_{e'}}$, $P^e$ is quasi-random in most faces in $C_{e,[k]}\times X_{k+1}$ (i.e.~the total measure of such faces is almost full $\mu(C_{e,[k]}\times X_{k+1}) $).

By assumption, most cells $C\subseteq C_{[k]}\times X_{k+1}$ satisfy: $P^e$ is sufficiently quasi-random in the face $C_e$ of $C$ \emph{for all $e \in \binom{[k]}{d}$ simultaneously}, i.e.~the measure of the union of all cells $C\subseteq C_{[k]}\times X_{k+1}$ for which it holds is  almost $\mu(C_{[k]}\times X_{k+1})$. Indeed,   for each $e \in \binom{[k]}{d}$, the set of bad faces $C_e$ has measure proportional to $C_{e,[k]}$ by the previous paragraph; but, because of the quasi-randomness of those $A_{e',i_{e'}}$ with $e'$ not a subset of $e$, the set of extensions of $C_e$ to a cell extending $C_{[k]}$ has measure proportional to $C_{[k]}$.  Fix such a cell $C\subseteq C_{[k]}\times X_{k+1}$.

  Fix an ordering $\binom{[k]}{d}=\left\{e_1,\ldots,e_{\binom{k}{d}}\right\}$ and consider any $1 \leq j \leq \binom{k}{d}$. Observe that each fixed  $\vec x_{[k]\setminus e_j} \in \prod_{i \in [k]\setminus e_j} X_i$ gives a sub-cylinder intersection set of the face $C_{e_j}$ (note that $|e_j \cap e_{j'}| \leq d$ for all $j' \neq j$):
  \begin{gather*}
  	B_{e_j,\vec x_{[k]\setminus e_j}} := \\
  \bigcap_{e'\subsetneq e_j\cup\{k+1\}}\left(A_{e',i_{e'}}\cap \bigcap_{e^*\subseteq[k]\setminus e_j, | e' \cup e^*|\leq d} (A_{e'\cup e^*,i_{e'\cup e^*}})_{\vec x_{e^*}}\right) \  \cap \  \bigcap_{j'>j}P^{e_{j'}}_{\vec x_{(e_{j'} \setminus e_j) }}.
  \end{gather*}
  For example, consider the case where $k=2$, $d=1$, and $e_1=\{1\}$. A face is just a set $A_{1,i_1}\times A_{3,i_3}$ where $A_{1,i_1} \subseteq X_1$ and $A_{3,i_3}\subseteq X_3 = \Omega$. When we fix a value $x_2 \in X_2$, we get a set $B_{\{1\}, x_2} = A_{1,i_1}\times (A_{3,i_3}\cap P^{\{2\}}_{x_2})$; the quasi-randomness of $P^{\{1\}}$ in $A_{1,i_1}\times A_{3,i_3}$ lets us approximate the density of $P^{\{1\}}$ in $A_{1,i_1}\times (A_{3,i_3}\cap P^{\{2\}}_{x_2})$:
  
  $\frac{\mu \left( P^{\{1\}} \cap \left(A_{1,i_1}\times (A_{3,i_3}\cap P^{\{2\}}_{x_2}) \right) \right)}{\mu \left( A_{1,i_1}\times (A_{3,i_3}\cap P^{\{2\}}_{x_2}) \right)} \approx \gamma_{\{1,3\}} $.  And for any fixed $x_1$  we have $\left( P^{\{1\}} \cap B_{e_1,x_2}  \right)_{x_1} = \left( (P^{\{1\}} \cap P^{\{2\}}) \cap (A_{1,i_1} \times A_{2,i_2} \times A_{3,i_3}) \right)_{(x_1,x_2)}$, so for most $x_1 \in A_{1,i_1}$ we have $\mu ( (P^{\{1\}} \cap P^{\{2\}}) \cap C)_{(x_1,x_2)}) = \gamma_{\{1,3\}} \mu((B_{\{1\}, x_2})_{x_1})$.

  For a more complicated example, take $k=4$, $d=2$, and $e_1=\{1,2\}$. A $\{1,2\}$-face is an intersection of six sets of the form $\bigcap_{e\subsetneq\{1,2,5\}}A_{e,i_e}$. When we fix some $(x_3,x_4)\in X_3\times X_4$, we take a subset out of each of these five  sets: $A_{1,i_{1}}\cap \left(A_{\{1,3\},i_{\{1,3\}}} \right)_{x_3}\cap \left(A_{\{1,4\},i_{\{1,4\}}} \right)_{x_4}$ is a subset of $A_{1,i_1}$, and we get similar subsets of $A_{2,i_2}$ and $A_{5,i_5}$; and $A_{\{1,5\},i_{\{1,5\}}}\cap P^{\{1,3\}}_{x_3}\cap P^{\{1,4\}}_{x_4}$ is a subset of $A_{\{1,5\},i_{\{1,5\}}}$, and similarly for $A_{\{2,5\},i_{\{2,5\}}}$. Note that the subsets we are taking are slices of sets ``one level up''---we take subsets of the binary sets using slices of the (ternary) $P^{e}$, then subsets of the unary sets using slices of the binary sets from our partitions.

  More generally, we will take subsets of $d$-ary sets using slices of the $P^e$, then subsets of the $(d-1)$-ary sets using slices of the $d$-ary sets, and so on. This is precisely the situation that hypergraph regularity lets us control: the quasi-randomness of $P^{e_j}$ in $C_{e_j}$ tells us that $\frac{\mu \left(P^{e_j}\cap B_{e_j,\vec x_{[k]\setminus e_j}} \right)}{\mu \left(B_{e_j,\vec x_{[k]\setminus e_j}} \right)} \approx \gamma_{C_{e_j}}$, and for most $\vec x_{e_j}\in (C_{e_j, [k]})_{\vec x_{[k]\setminus e_j}}$ we have:
  \begin{align*}
    \mu((\bigcap_{j'\geq j}P^{e_{j'}}\cap C)_{\vec x})
    &=\mu((P^{e_j}\cap B_{e_j,\vec x_{[k]\setminus e_j}})_{\vec x_{e_j}}) \\
    &\approx \gamma_{C_{e_j}}\mu((B_{e_j,\vec x_{[k]\setminus e_j}})_{\vec x_{e_j}})\\
    &=\gamma_{C_{e_j}}\mu((\bigcap_{j'>j}P^{e_{j'}}\cap C)_{\vec x}),
    \end{align*}
where $\vec{x} \in \prod_{i \in [k]} X_i$ is the tuple agreeing with $\vec x_{e_j}$ and  $\vec x_{[k]\setminus e_j}$ on the corresponding coordinates, and in the first and third equalities the underlying sets are equal by definition of $B_{e_j,\vec x_{[k]\setminus e_j}}$. Let us say that the pair $(\vec x_{[k]\setminus e_j}$, $\vec x_{e_j})$ is \emph{well-behaved in $C$} if $\mu((\bigcap_{j'\geq j}P^{e_{j'}}\cap C)_{\vec x}) \approx \gamma_{C_{e_j}}\mu((\bigcap_{j'>j}P^{e_{j'}}\cap C)_{\vec x})$.

  We have shown that, for each $e_j$ with $j \leq \binom{k}{d}$, for every $\vec x_{[k]\setminus e_j}$, most $\vec x_{e_j}$ in $(C_{e_j,[k]})_{\vec x_{[k]\setminus e_j}}$ are well-behaved in $C$. Using quasi-randomness of the $A_{e',i_{e'}}$, this implies that most $\vec x\in C_{[k]}$ have the property that, for all $e_j$ simultaneously, the pair $(\vec x_{[k]\setminus e_j}, \vec x_{e_j})$ is well-behaved in $C$. Therefore we may successively apply being well-behaved for each $j \leq  \binom{k}{d} $ to get that: most $\vec{x} \in C_{[k]}$ satisfy
  \[\mu((\bigcap_{e \in   \binom{[k]}{d}} P^e\cap C)_{\vec x})\approx \prod_{e \in   \binom{[k]}{d}} \gamma_{C_e}\mu(C_{\vec x}).\]

Unfixing $C$,  we thus have shown that for most cells $C\subseteq C_{[k]}\times X_{k+1}$, most $\vec x\in C_{[k]}$ are well-behaved at $C$. Hence also most $\vec x\in C_{[k]}$ are well-behaved at most cells $C\subseteq C_{[k]}\times X_{k+1}$ simultaneously.

  Additionally, hypergraph regularity ensures that, for most cells $C\subseteq C_{[k]}\times X_{k+1}$ and most $\vec x \in C_{[k]}$, there is a $\delta_C$ so that $\mu(C_{\vec x})\approx \delta_C$. Therefore, again, for most $\vec x \in C_{[k]}$, for most $C \subseteq C_{[k]}\times X_{k+1}$ simultaneously we have $\mu(C_{\vec x})\approx\delta_C$. Therefore for most $\vec x\in C_{[k]}$, we have
  \begin{align*}
    \mu(\bigcap_{e \in \binom{[k]}{d}} P^e_{\vec x_e})
    &=\sum_{C \subseteq C_{[k]}\times X_{k+1}}\mu((\bigcap_{e} P^e\cap C)_{\vec x})\\
    &\approx \sum_C\prod_e \gamma_{C_e}\mu(C_{\vec x})\\
    &\approx\sum_C\prod_e \gamma_{C_e}\delta_C,
  \end{align*}
  concluding the proof.
\end{proof}

We generalize Theorem \ref{thm:main} from measures of intersections of sets to integrals of continuous combinations of arbitrary real valued functions:

\begin{cor}\label{cor: part for integrals}
	 For every $1 \leq d < k \in \mathbb{N}$, continuous function $h: [0,1]^{\binom{k}{d} } \to [0,1]$, $\varepsilon \in \mathbb{R}_{>0}$ and function $F: \mathbb{N} \to (0,1]$  there exists $N = N(d,k,h, \varepsilon, F) \in \mathbb{N}$ satisfying the following.
 
  Let finite sets $\{X_i\}_{1 \leq i\leq k}$ and a probability measure space $(\Omega,\mathcal{B},\mu)$ be given, and  for each $e\in\binom{[k]}{d}$ and each $\vec x\in\prod_{i\in e}X_i$, let $f^e_{\vec x}: \Omega \to [0,1]$ be a $\mathcal{B}$-measurable function. 
  Then the conclusion of Theorem \ref{thm:main} holds for the function $f:\prod_{1 \leq i\leq k}X_i\rightarrow[0,1]$ given by 
$f(\vec x) := \int_{\Omega} h \left( \left( f^e_{\vec x_e}(z) : e\in\binom{[k]}{d} \right)  \right) d \mu(z)$.
\end{cor}
\begin{proof}

Let $d,k, h, \varepsilon, F$ be given. For a function $g: \Omega \to [0,1]$, we write $g^{[\alpha, \beta)} = \{z \in \Omega : \alpha \leq g(z) < \beta \}$, and $\chi_S$ is the characteristic function of a set $S$.

Given $\delta_1$, we can choose a sufficiently large $n = n(\delta_1) \in \mathbb{N}$ satisfying the following. For each $e\in\binom{[k]}{d}$ and $\vec{x} \in \prod_{i\in e}X_i$, consider the simple function 
$$s^e_{\vec{x}}(z) := \sum_{i=1}^{4^n} \left( \frac{i}{2^n} \chi_{(f^e_{\vec{x}})^{[i/2^n, (i+1)/2^n)}} (z) \right);$$
\noindent then for all $e \in \binom{[k]}{d}$, for all $\vec{x} \in \prod_{i\in e}X_i$ we have: $|f^e_{\vec{x}}(z) - s^e_{\vec{x}}(z)| \leq \delta_1$ for all $z \in \Omega$.

If $\delta_1 = \delta_1(\varepsilon, k,d) \ll \varepsilon$, by uniform continuity of $h$  we then have: for all $\vec{x} \in \prod_{i\in [k]}X_i$, 
\begin{gather*}
	\left \lvert h \left( \left( f^e_{\vec x_e}\right)_{e\in\binom{[k]}{d} } \right)(z) -  h \left( \left( s^e_{\vec x_e} \right)_{e\in\binom{[k]}{d} }  \right)(z) \right \rvert \leq \varepsilon/5 \textrm{ for all } z \in \Omega \textrm{, so }\\
	\left \lvert \int_{\Omega} h \left( \left( f^e_{\vec x_e}\right)_{e\in\binom{[k]}{d} } \right) dz -  \int_{\Omega} h \left( \left( s^e_{\vec x_e}\right)_{e\in\binom{[k]}{d} } \right) dz  \right \rvert \leq  \varepsilon/5.
\end{gather*}

Let $g_{\vec{x}}(z) := h \left( \left( s^e_{\vec x_e}(z)\right)_{e\in\binom{[k]}{d} } \right)$. We can choose a sufficiently large $n' = n'(\varepsilon) \in \mathbb{N}$ so that: for all $\vec{x} \in \prod_{1 \leq i\leq k}X_i $, 
	$$\left \lvert \int_{\Omega} g(\vec{x},z) dz - \sum_{i=1}^{4^{n'}} \frac{i}{2^{n'}} \mu\left( g^{[i/2^{n'}, (i+1)/2^{n'})}_{\vec{x}} \right) \right \rvert \leq \varepsilon/5.$$

Note that for any $\vec{x}$ and any choice of $\vec{i} = (i_e : e \in \binom{[k]}{d}) \in [4^n]^{\binom{k}{d}}$, the tuple $\left( s^e_{\vec x_e}(z)\right)_{e\in\binom{[k]}{d} }$, and hence $g_{\vec{x}}(z)$, takes constant value on the set $\bigcap_{e} (f^e_{\vec{x}_e})^{[\frac{i_e}{2^n}, \frac{i_e+1}{2^n})}$. Hence for every $i \in [4^{n'}]$ there is some $J_i \subseteq [4^n]^{\binom{k}{d}}$ so that: for all $\vec{x}$,
\begin{gather*}
	g^{[i/2^{n'}, (i+1)/2^{n'})}_{\vec{x}} = \bigsqcup_{(i_e : e \in \binom{[k]}{d}) \in J_i} \bigcap_{e} (f_e)^{[i_e/2^n, (i_e+1)/2^n)}_{\vec{x}_e} \textrm{, so}\\
	\mu \left( g^{[i/2^{n'}, (i+1)/2^{n'})}_{\vec{x}} \right)  = \sum_{(i_e : e \in \binom{[k]}{d}) \in J_i} \mu \left( \bigcap_{e} (f_e)^{[i_e/2^n, (i_e+1)/2^n)}_{\vec{x}_e} \right).
\end{gather*}

Given arbitrary $\varepsilon'$ and applying Theorem \ref{thm:main} to each of the functions $\vec{x} \mapsto  \mu \left( \bigcap_{e} (f_e)^{[i_e/2^n, (i_e+1)/2^n)}_{\vec{x}_e} \right)$ with $F' = F'(F, n) = F'(F,\varepsilon)  \ll F, \varepsilon'' = \varepsilon''(\varepsilon',n) \ll \varepsilon'$ and taking a common refinement of the resulting partitions (using Remark \ref{rem: 2-stab reg ignore small cells}),
 we  find partitions $(S_{e,i} : 0 \leq i \leq b_e)$ with all $b_e \leq N = N(F',\varepsilon'')$ of $\prod_{i \in e} X_i$ so that for each cylinder intersection $C = \cap_{e} S_{e,i_e}$ with $i_e >0$ there are intervals $I_{\vec{i}} \subseteq [0,1]$ for $\vec{i} \in [4^n]^{\binom{k}{d}}$ of length $ < \varepsilon'$  and $Z \subseteq C$ with $|Z| \leq  F(b) |C|$ so that $\mu \left( \bigcap_{e} (f_e)^{[i_e/2^n, (i_e+1)/2^n)}_{\vec{x}_e} \right) \in I_{(i_e : e \in \binom{[k]}{d}) }$ for all $\vec{i}  \in [4^n]^{\binom{k}{d}}$  and all $\vec{x} \in C \setminus Z$. We can choose $\varepsilon' = \varepsilon'(\varepsilon)$ so that $4^{n'}(4^n)^{\binom{k}{d}} \varepsilon' \leq \varepsilon/5$, then we get that the value of 
 $\sum_{i=1}^{4^{n'}} \frac{i}{2^{n'}} \mu\left( g^{[i/2^{n'}, (i+1)/2^{n'})}_{\vec{x}} \right)$ belongs to some fixed interval of length $\leq \varepsilon/5$ for all $\vec{x} \in C \setminus Z$. Hence the value of $\int_{\Omega} h \left( \left( s^e_{\vec x_e}\right)_{e\in\binom{[k]}{d} } \right) dz  $ belongs to some fixed interval of length $\leq \varepsilon$ for all $\vec{x} \in C \setminus Z$, as required.
\end{proof}

\section{No Single Excluded Substructure for strong $2$-stability}\label{sec:no_contain}
%
%
\subsection{Higher arity stability and perfect regularity lemmas}\label{sec: ultraprod and perf reg}

Assume $k \in \mathbb{N}$ is fixed and we are given \emph{base sets}  $X_1, \ldots, X_k$. By a \emph{sort} we will mean an arbitrary cartesian product $X$ of the sets $X_1, \ldots, X_k$, in an arbitrary order and possibly with repetitions. 
A \emph{$k$-partite graded probability space} on the base sets $X_1, \ldots, X_k$ is given by specifying, for each sort $X$, a $\sigma$-algebra $\mathcal{B}_{X} \subseteq \mathcal{P}(X)$ of subsets of $X$ and a countably additive probability measure $\mu_{X}$ on $\mathcal{B}_{X}$ satisfying certain compatibility conditions between different sorts including closure under products and Fubini (we refer to \cite[Sections 2.2]{chernikov2020hypergraph} and \cite[Sections 2.2]{chernikov2024perfect} for a further discussion of graded probability spaces and their basic properties). 

Given a sequence of $k$-ary functions on products of $k$ finite sets of the form $(f_i; X_{1,i}, \ldots, X_{k,i})$ with $X_{u,i}$ finite and $f_i: \prod_{u \in [k]} X_{u,i} \to [0,1]$ for $i \in \mathbb{N}$ and $\mu_{e,i}$ a probability measure on $\prod_{u \in e} X_{u,i}$, we study its asymptotic behavior by considering its ultraproducts that give rise to  measurable functions on the graded probability spaces generated by internal subsets. Formally, we let $\mathcal{M}_i$ be a first-order structure in a countable language $\mathcal{L}$ so that $X_{u,i}, u \in [k]$ are definable sets in it, and for every $q \in \mathbb{Q}$ its signature contains a $k$-ary relation symbol $F^{<q}$ so that for any $(x_1, \ldots, x_k) \in \prod_{u \in [k]} X_{u,i}$, $\mathcal{M}_{i} \models F^{<q} \iff f_i(x_1, \ldots, x_k) < q$. Given $\mathcal{U}$ a non-principal ultrafilter on $\mathbb{N}$, we consider the ultraproduct $\mathcal{M}^{\ast} = (X_1, \ldots, X_k, f, \ldots) := \prod_{i \in \mathbb{N}} \mathcal{M}_i / {\mathcal{U}}$, where $X_u$ is the definable set given by the ultraproduct of the $X_{u,i}$'s, and $f: \prod_{u \in [k]} X_u \to [0,1]$ is defined via $f(x_1, \ldots, x_k) = \inf \{ q \in \mathbb{Q} : \mathcal{M}^{\ast} \models F^{<q}(x_1, \ldots, x_k)\}$. For $e \subseteq [k]$, we let $\mathcal{B}_{e,0}$ be the Boolean algebra of all internal subsets of $\prod_{u \in e} X_u$ (i.e.~sets that are ultraproducts of subsets of $\prod_{u \in e} X_{u,i}$), and let $\mathcal{B}_{e}$ be the $\sigma$-algebra generated by $\mathcal{B}_{e,0}$; then $f$ is $\mathcal{B}_{e}$-measurable. For $Y = \prod_{i \in \mathbb{N}} Y_i/\mathcal{U}  \in \mathcal{B}_{e,0}$, we let $\mu_{e}^{0}(Y) := \lim_{\mathcal{U}} \mu_{e,i}(Y_i)$. Then $\mu_{e}^{0}$ is a finitely additive probability measure on the Boolean algebra $\mathcal{B}_{e,0}$, and by Carath\'eodory's extension theorem it extends uniquely to a countably additive probability measure $\mu_{e}$ on $\mathcal{B}_{e}$. 	 Given an internal set $Y \subseteq \prod_{u \in e} X_u$ with $Y = \prod_{i \in \mathbb{N}} Y_i / \mathcal{U}$ and $r \in \mathbb{R}$,  if $ \mu_{e,i}(Y_i) \leq r$ for a $\mathcal{U}$-large set of $i \in \mathbb{N}$ then $\mu_e(Y) \leq r$; conversely,  if $ \mu_{e,i}(Y_i) \leq r$ then for any $\varepsilon > 0$ there is a $\mathcal{U}$-large set  of $i \in \mathbb{N}$ with $ \mu_{e,i}(Y_i) \leq r + \varepsilon$. This implies that for any set $Y \in \mathcal{B}_e$ with $e = e_1 \sqcup e_2$ and $r \in [0,1]$, the set $\{y \in \prod_{u \in e_2} X_u : \mu_{e_1}(Y_{y}) \square r\}$  for $\square \in \{<, \leq, =, >, \geq \}$ is in $\mathcal{B}_{e_2}$; and Fubini theorem is satisfied, so $(X_u, \mathcal{B}_e, \mu_e)$ form a graded probability space.

Without loss of generality (expanding the signature $\mathcal{L}$ of $\mathcal{M}_i$ if necessary), we may assume that our measures are \emph{definable}: for every formula $\varphi(\bar{y}, \bar{w}) \in \mathcal{L}$ with $\bar{y} = (y_u : u \in e)$ for $e \subseteq [k]$ a tuple of variables corresponding to $X_u$ and $r \in \mathbb{Q}_{>0}$, we have a formula $m_{\bar{y}} < r. \varphi(\bar{y}, \bar{w})$ in $\mathcal{L}$ so that:
	 \begin{itemize}
	 	\item for all $i \in \mathbb{N}$, $\bar{b}$ tuple in $\mathcal{M}_i$, $\mathcal{M}_i \models m_{\bar{y}} < r. \varphi(\bar{y}, \bar{b}) \Leftrightarrow  \mu_{e,i}(\varphi(\mathcal{M}_i, \bar{b})) < r$, where $\varphi(\mathcal{M}_i, \bar{b}) = \{ \bar{a} \in \prod_{u \in e}X_{u,i}: \mathcal{M}_i \models \varphi(\bar{a},\bar{b}) \}$;
	 	\item for any $\bar{b}$ in $\mathcal{M}^{\ast}$ and $\varepsilon > 0$, $\mu_{e}(\varphi(\mathcal{M}^{\ast}, \bar{b})) < r  \Rightarrow \mathcal{M}^{\ast} \models m_{\bar{y}} < r. \varphi(\bar{y}, \bar{b}) \Rightarrow \mu_{e}(\varphi(\mathcal{M}^{\ast}, \bar{b})) < r + \varepsilon$.
	 \end{itemize}

	 We refer to \cite[Section 9.3]{chernikov2020hypergraph} and \cite[Section  2.3]{chernikov2024perfect} for a  more detailed discussion of ultraproducts of graded probability spaces.

\begin{definition}\label{def: perf 2-stab reg}
\begin{enumerate}
	\item For $1 \leq d \leq k$ and a graded probability space on $X_1, \ldots, X_k$, we  say that
	a $\mathcal{B}_{[k]}$-measurable function $f : \prod_{i \in [x]} X_i \to [0,1]$ satisfies \emph{perfect $d$-stable regularity} if,  for every $\varepsilon \in \mathbb{R}_{>0}$  there exist countable partitions  $\left\{ A_{e,i} : i  \in \mathbb{N} \right\}$ of $\prod_{i \in e} X_i$  with $A_{e,i} \in \mathcal{B}_{e}$ for every  $e \in \binom{[k]}{d}$ so that:  for every $(i_e \in \mathbb{N} : e \in \binom{[k]}{d})$ there is some interval $I \subseteq [0,1]$ of length $|I| \leq  \varepsilon$ so that $f(\vec{x}) \in I$ for all $\vec{x}$ in the cylinder intersection set $ \bigcap_{e \in \binom{[k]}{d}} A_{e, i_e}$ outside of a measure $0$ subset.

	If $E \in \mathcal{B}_{[k]}$ is a hypergraph, we identify it with its characteristic function $f := \chi_{E}$, hence perfect $d$-stable regularity for $E$ requires that the density of $E$ on each $\bigcap_{e \in \binom{[k]}{d}} A_{e, i_e}$ is either $0$ or $1$.  
	
	\item We say that a family $\mathcal{F}$ of $k$-ary functions on products of $k$ finite sets of the form $(f; X_1, \ldots, X_k)$ with $X_i$ finite and $f: \prod_{i \in [k]} X_i \to [0,1]$ satisfies \emph{perfect $d$-stable regularity} if for any sequence of functions from $\mathcal{F}$, any choice of the measures $\mu_{u,i}, u \in [k]$ and any ultrafilter $\mathcal{U}$ on $\mathbb{N}$, the associated ultraproduct satisfies (1).

\item We say that a family $\mathcal{F}$ of $k$-ary functions on products of $k$ finite sets of the form $(f; X_1, \ldots, X_k)$ with $X_i$ finite and $f: \prod_{i \in [k]} X_i \to [0,1]$ satisfies \emph{strong $d$-stable regularity} if it satisfies the conclusion of Theorem \ref{thm:main}.

\end{enumerate}
\end{definition}

\begin{prop}\label{prop: perf implies strong}
	If a family $\mathcal{F}$ of $k$-ary functions on products of $k$ finite sets of the form $(f; X_1, \ldots, X_k)$ with $X_i$ finite and $f: \prod_{i \in [k]} X_i \to [0,1]$ satisfies perfect $d$-stable regularity, then $\mathcal{F}$ satisfies strong $d$-stable regularity. The converse is not true, already for $\{0,1\}$ functions, $k=3$ and $d=1$.
	
\end{prop}
\begin{proof}
A proof for $d=1$ and $\{0,1\}$-functions is given in \cite[Remark 4.4]{chernikov2024perfect}, and it generalizes in a straightforward manner to real valued functions and higher $d$. 
	
	Indeed, assume strong $d$-stable regularity fails for $\mathcal{F}$ with a given $F: \mathbb{N} \to [0,1]$ and $\varepsilon >0$. I.e.~for every $i \in \mathbb{N}$ we have some $(f_i; X_{1,i}, \ldots, X_{k,i})$ with $X_{u,i}$ finite, probability measures $\mu_{e,i}$ and $f_i: \prod_{u \in [k]} X_{i,u} \to [0,1]$ so that $N := i$ does not satisfy the conclusion of Theorem \ref{thm:main} with respect to $\varepsilon$ and $F$. Consider the ultraproduct $X_u, f$ and the associated graded probability space $\mathcal{B}_e, \mu_e$.
	  
	  By assumption there exist countable partitions  $\left\{ A_{e,t} : t  \in \mathbb{N}_{\geq 1} \right\}$ of $\prod_{u \in e} X_u$  with $A_{e,t} \in \mathcal{B}_{e}$ for every  $e \in \binom{[k]}{d}$ so that:  for every tuple $(t_e \in \mathbb{N}_{\geq 1} : e \in \binom{[k]}{d})$ there is some interval $I \subseteq [0,1]$ of length $|I| \ll \varepsilon$ so that $f(\vec{x}) \in I$ for all $\vec{x} \in \bigcap_{e \in \binom{[k]}{d}} A_{e, t_e}$ outside of a measure $0$ subset. First,  we can choose finite $k_e \in \mathbb{N}$ so that the union of $\{ A_{e,t} : 1 \leq t \leq k_e \}$ has measure $\gg 1 - \varepsilon'$, let $k := \max_{e} k_e$. As $\mathcal{B}_e$ is the $\sigma$-algebra generated by $\mathcal{B}_{e,0}$, we can choose pairwise disjoint internal sets $A'_{e,t}$ so that 
	   $\mu( A_{e,t} \triangle A'_{e,t}) < \delta$    for all $1 \leq t \leq k_e$, where  we can take $\delta >0$ arbitrarily small. Let $\delta$ be much smaller than $F(k) $ times the minimum of the positive values of $\mu(\bigcap_{e \in \binom{[k]}{d}} A_{e, t_e})$ over all tuples $(t_e)$  with $1 \leq t_e \leq k$. Then we still have $f(\vec{x}) \in I$ for all $\vec{x} \in \bigcap_{e \in \binom{[k]}{d}} A'_{e, t_e}$ outside of a set of measure $\ll F(k) \mu(\bigcap_{e \in \binom{[k]}{d}} A'_{e, t_e})$, for all such tuples $(t_e)$. And letting $A'_{e,0}$ be the complement of the union of the $A'_{e,t}$'s we have $\mu(A'_{e,0}) \ll \varepsilon$. By \L os' theorem and definition of ultralimits, for a $\mathcal{U}$-large set of indices $i \in \mathbb{N}$ we thus have partitions $(A'_{e,t})_i$ of size $\leq k_e$ with $\mu_{e,i}((A'_{e,0})_i) \leq \varepsilon$ and $f_i(\vec{x}) \in I'$ for a slightly bigger interval with $|I'| \leq \varepsilon$ for all $\vec{x} \in \bigcap_{e \in \binom{[k]}{d}} (A'_{e, t_e})_i$ outside of a set of slightly bigger measure $\leq F(k) \mu(\bigcap_{e \in \binom{[k]}{d}} (A'_{e, t_e})_i)$, for all $t_e \geq 1$. But this contradicts the choice of $f_i$'s for $i > k$.
	   
	  Finally, an example demonstrating that perfect $1$-stable regularity is strictly stronger than $1$-stable regularity for $3$-hypergraphs is given in \cite[Example 4.16]{chernikov2024perfect}.
\end{proof}

In the definition of \emph{strong} $2$-stable regularity for $3$-hypergraphs we only require existence of partitions of pairs of vertices. This is because the  corresponding partitions of \emph{vertices} to fit with the usual definition of quasi-randomness for  $3$-hypergraphs can always be arranged a posteriori using standard arguments applying general graph regularity to each binary set in the given $2$-stable partition. We provide some details:

\begin{prop}\label{prop: str 2-stab reg implies disc}
For a family $\mathcal{H}$ of finite $3$-partite $3$-hypergraphs, strong $2$-stable regularity implies (the partite version of) \emph{binary $\disc_{2,3}$-error}, in the sense of \cite[Definition 2.11]{terry2021irregular}.

\end{prop}
\begin{proof}

Let $\varepsilon_1 \in \mathbb{R}_{>0}$ and $\varepsilon_2: \mathbb{N} \to (0,1]$ be arbitrary. 

	Given a sufficiently large $H = (E; X_1, X_2, X_3) \in \mathcal{H}$, by strong $2$-stable  regularity for $\varepsilon$ and monotone decreasing $F: \mathbb{N} \to (0,1]$ to be specified later, there exists $N = N(\varepsilon, F)$ and  partitions of pairs of vertices $\prod_{u \in e}X_u=\bigsqcup_{0 \leq \alpha \leq \ell_{e}}S_{e,\alpha}$ with $\ell_e \leq N$ for each $e\in\binom{[3]}{2}$ so that $|S_{e,0}| < \varepsilon |\prod_{u \in e}X_u|$ and for every non-empty cylinder intersection set $S=\bigcap_{e\in\binom{[3]}{2}}S_{e,\alpha_e}$ with each $\alpha_e>0$, there is 
	$Z\subseteq S$ with $|Z|<F(\ell) |S|$ so that all $\vec x \in 
	S \setminus Z$ are simultaneously in $E$, or simultaneously not in $E$, where $\ell := \max_{e} \ell_e$.
	
	Applying strong graph regularity with $\varepsilon_3 = \varepsilon_3(\varepsilon_4, \ell) \ll \varepsilon_4$  to each bipartite graph $S_{\{u,v\},\alpha} \subseteq X_{u} \times X_{v}$ with $1 \leq u < v \leq 3$ and $0 \leq \alpha \leq  \ell_{\{u,v\}}$, and taking a common refinement of all of these partitions (see e.g.~\cite[Lemma 3.28]{terry2021irregular}, \cite[Lemma 3.7]{frankl2002extremal}) on the same coordinate (and regrouping the parts to get an equipartition), for each $1 \leq u \leq 3$ we obtain an equipartition $\mathcal{P}_{u} = \{V^u_{i} : 0 \leq i \leq t_u \}$ of $X_u$ with $t_u \leq T = T(\varepsilon_3, \ell)$ so that: 
	for any $1 \leq u < v \leq 3$, $0 \leq \alpha \leq \ell_{\{u,v\}}$ and $(i,j) \in ([t_u] \times [t_v]) \setminus \Gamma'_{u,v}$ with $|\Gamma'_{u,v}| \leq \varepsilon_4 t_u t_v$, $S_{\{u,v\},\alpha}$ is $\varepsilon_4$-regular on $V^u_i \times V^v_j$ (i.e.~satisfies $\disc_2(\varepsilon_4)$ in the terminology of \cite[Definition 2.1]{terry2021irregular}). We let $\Gamma''_{u,v}$ be the set of all $(i,j) \in [t_u] \times [t_v]$ with $|S_{\{u,v\},0} \cap (V^u_i \times V^v_j)| \geq \varepsilon^{1/2} |V^u_i| |V^v_j|$. Since $|S_{\{u,v\},0}| < \varepsilon |X_u||X_v|$ we have $\Gamma''_{u,v} \leq \varepsilon^{1/2} t_u t_v$. Let $\Gamma_{u,v} := \Gamma'_{u,v} \cup \Gamma''_{u,v}$, then $|\Gamma_{u,v}| \leq (\varepsilon^{1/2} + \varepsilon_4) t_u t_v \leq \varepsilon_1 t_u t_v$, assuming $\varepsilon = \varepsilon(\varepsilon_1), \varepsilon_4 \ll \varepsilon_1$.
	
		For $e = \{u,v\}$,  $1 \leq u < v \leq 3$, $(i,j) \in [t_u] \times [t_v]$ and $0 \leq \alpha \leq  \ell_{e}$, we let $Q^{e,\alpha}_{i,j} := S_{e,\alpha} \cap (V^u_i \times V^v_j)$.

	Fix  $(i,j) \notin \Gamma_e$ and let $d^{e,\alpha}_{i,j}$ be the density of $Q^{e, \alpha}_{i,j}$ on $V^u_i \times V^v_j$, then $\sum_{0 \leq \alpha \leq \ell_e} d^{e,\alpha}_{i,j} = 1$. As $(i,j) \notin \Gamma_e$, we have $d^{e,0}_{i,j} \leq \varepsilon^{1/2}$. Let $J^{e}_{i,j}$ be the set of all $1 \leq \alpha \leq \ell_e$ with $d^{e,\alpha}_{i,j} <  \varepsilon^{1/2} / \ell_e$. Then $d^e_{i,j} := \sum_{\alpha \in \{0\} \cup J^{e}_{i,j}} d^{e,\alpha}_{i,j} \leq 2 \varepsilon^{1/2}$.
	
	Using that we can subdivide and take unions of sufficiently quasi-random graphs preserving enough quasi-randomness (see e.g.~\cite[Fact 3.24, Lemma 3.25]{terry2021irregular}), we can  distribute $Q^{e}_{i,j} := \bigcup_{\alpha \in \{0\} \cup J^{e}_{i,j}} Q^{e, \alpha}_{i,j}$ proportionally among  the sets $Q^{e,\alpha}_{i,j}$ with $\alpha \in [\ell_e] \setminus  J^{e}_{i,j}$, maintaining quasirandomness. Namely, for each  $\alpha \in [\ell_e] \setminus  J^{e}_{i,j}$, we define $P^{e, \alpha}_{i,j}$ by adding to $Q^{e, \alpha}_{i,j}$ at most 
	$$\frac{2 d^{e}_{i,j}}{\sum_{\alpha \in [\ell_e] \setminus  J^{e}_{i,j}} d^{e,\alpha}_{i,j}} d^{e,\alpha}_{i,j} |V_i||V_j| \leq \frac{4 \varepsilon^{1/2}}{1-2\varepsilon^{1/2}}d^{e,\alpha}_{i,j} |V_i||V_j| $$
	points from $Q^{e}_{i,j}$, obtaining a new partition $\{P^{e, \alpha}_{i,j} :\alpha \in [\ell_e] \setminus  J^{e}_{i,j} \}$ of $V^u_i \times V^v_j$ of size $\leq \ell_e$, so that still each $P^{\varepsilon, \alpha}_{i,j}$ (we let $\tilde{d}^{e, \alpha}_{i,j}$ denote its density) satisfies $\disc_2(\varepsilon_5)$, using $\varepsilon_4 = \varepsilon_4(\varepsilon_5, \ell, \varepsilon) \ll \varepsilon_5$.

Given $(i,j,k)$	with $(i,j) \notin \Gamma_{\{1,2\}}, (i,k) \notin \Gamma_{\{1,3\}}, (j,k) \notin \Gamma_{\{2,3\}}$ and \emph{arbitrary} $ \alpha \in [\ell_{\{1,2\}}] \setminus  J^{\{1,2\}}_{i,j} , \beta \in [\ell_{\{1,3\}}] \setminus  J^{\{1,3\}}_{i,k} , \gamma \in [\ell_{\{2,3\}}] \setminus  J^{\{2,3\}}_{j,k} $, consider the cylinder intersection set $C' := Q^{\{1,2\},\alpha}_{i,j} \cap Q^{\{1,3\},\beta}_{i,k} \cap Q^{\{2,3\},\gamma}_{j,k}$, and let $C$ be defined in the same way but with $P$  sets instead of $Q$ sets. Let $S := S_{\{1,2\},\alpha} \cap S_{\{1,3\},\beta} \cap S_{\{2,3\},\gamma}$, and let $Z \subseteq S$ with $|Z| <F(\ell) |S|$ be the error set as above. As $\varepsilon_5 = \varepsilon_5(\varepsilon_6) \ll \varepsilon_6$, $Q$'s and $P$'s satisfy  $\disc_2(\varepsilon_5)$ and $d^{\{1,2\},\alpha}_{i,j}, d^{\{1,3\},\beta}_{i,k}, d^{\{2,3\},\gamma}_{j,k} \geq \varepsilon^{1/2} / \ell$, by the counting lemma we have 
\begin{gather*}
	|C'| \geq \left( d^{\{1,2\}, \alpha}_{i,j} d^{\{1,3\}, \beta}_{i,k} d^{\{2,3\}, \gamma}_{j,k}  - \varepsilon_6 \right) |V^1_{i}| |V^2_{j}| |V^3_{k}| \geq (\varepsilon^{3/2} / \ell^3 - \varepsilon_6) |V^1_{i}| |V^2_{j}| |V^3_{k}|,\\
		|C| \leq \left( \left(1 + \frac{4 \varepsilon^{1/2}}{1-2\varepsilon^{1/2}} \right)^3 d^{\{1,2\}, \alpha}_{i,j} d^{\{1,3\}, \beta}_{i,k} d^{\{2,3\}, \gamma}_{j,k}  + \varepsilon_6 \right) |V^1_{i}| |V^2_{j}| |V^3_{k}|.
\end{gather*}
	If we started with the function $F$ so that $F(\ell) \ll (\frac{\varepsilon''_1}{\ell T (\varepsilon_3, \ell)} )^3$, 
	\begin{gather*}
		|S \setminus Z| \leq F(\ell) |S| \leq F(\ell) t_1 t_2 t_3 |V^1_i| |V^2_j| |V^3_k|  \leq \varepsilon''_1 |C'|.
	\end{gather*}
	
\noindent  As $C' \subseteq S$, if $\varepsilon''_1 = \varepsilon''_1(\varepsilon'_1) \ll \varepsilon'_1, \varepsilon_6 = \varepsilon_6(\varepsilon'_1, \ell, \varepsilon) \ll \varepsilon^{3/2} / \ell^3$ combining we get
	\begin{gather*}
		|C' \setminus Z| \geq (1-\varepsilon'_1) |C|.
	\end{gather*}
	
	As $C' \setminus Z$ is $E$-homogeneous, $\varepsilon'_1 = \varepsilon'_1(\varepsilon_1), \varepsilon (\varepsilon_1) \ll \varepsilon_1$, this implies that for any sub-cylinder intersection set $C_0$ of $C$,
	\begin{gather*}
		||E \cap C_0| - d |C_0|| \leq \varepsilon_1 d^{\{1,2\}, \alpha}_{i,j} d^{\{1,3\}, \beta}_{i,k} d^{\{2,3\}, \gamma}_{j,k} |V^1_{i}| |V^2_{j}| |V^3_{k}|,
	\end{gather*}
	where $d$ is the density of $E$ on $C$ (so close to $0$ or to $1$). Assuming $\varepsilon_6 \leq \varepsilon_2(\ell)$, it follows that the triad given by $C$ and $V^1_i, V^2_j, V^3_k$ satisfies (the partite version of) $\disc_{2,3}(\varepsilon_1, \varepsilon_2)$ (in the sense of \cite[Definition 2.7]{terry2021irregular}).

Finally, for fixed $i,j$, subdividing  and regrouping the sets $\{P^{e,\alpha}_{i,j} : \alpha \in [\ell_{e}]\setminus J^{e}_{i,j}\}$ (see the proof of \cite[Lemma 3.10]{frankl2002extremal}), we may assume that the partition $\{ P^{e,\alpha}_{i,j} : \alpha \leq \ell_e  \}$ of $V^u_i \times V^v_j$ is equitable (i.e.~$P^{e,\alpha}_{i,j}$ satisfy $\disc_2(\varepsilon_2, 1/\ell_e)$, as required in \cite[Definition 2.6]{terry2021irregular}). Note that all of the appropriate functions $\varepsilon_i$ and $F$ can be chosen starting with given $\varepsilon_1, \varepsilon_2$.
%
%
%
\end{proof}

\begin{remark}\label{rem: 2-stab reg ignore small cells}
	In the definition of strong $d$-stable regularity, we may additionally require that the partition satisfies: for every cylinder intersection $C = \bigcap_{e} S_{e,i_e}$ with $i_e >0$ for all $e$,  $\mu(C) \geq F(b) \prod_{e} \mu(S_{e,i_e})$; and $\mu(S_{e,i_e}) \geq F(b)$, where $b = \max_{e} b_e$.
\end{remark}
\begin{proof}
Let $\varepsilon'$ and $F': \mathbb{N} \to (0,1]$ be arbitrary.

	Start with a partition given by $2$-stable regularity for $\varepsilon \ll \varepsilon'$ and $F \ll F'$. We follow the notation in the proof of Proposition \ref{prop: str 2-stab reg implies disc}.
	First, for $e = \{u,v\}$, $u < v$, we replace $S_{e,0}$ by $S'_{e,0}$ obtained by adding to $S_{e,0}$ the union $\bigcup_{(i,j) \in \Gamma'_e} V_i \times V_j$ and $S_{e,i}$ for all $i$ with $\mu(S_{e,i}) \ll F(b) \ll \varepsilon/b$, then still $\mu(S'_{e,0}) \leq \varepsilon' = \varepsilon'(\varepsilon)$. 
			
Given $(i,j) \notin \Gamma'_e$, we let $(J')^{e}_{i,j}$ be the set of all $1 \leq \alpha \leq \ell_e$ with $d^{e,\alpha}_{i,j} <  F(\ell) / \ell$. For all $\alpha \in [\ell_e] \setminus  (J')^{e}_{i,j}$, we distribute $(Q')^{e}_{i,j} := \bigcup_{\alpha \in (J')^{e}_{i,j}} Q^{e, \alpha}_{i,j}$ (so we are not including $S_{e,0} \cap V_i \times V_j$ anymore like in the proof of Proposition \ref{prop: str 2-stab reg implies disc}) proportionally among  the sets $Q^{e,\alpha}_{i,j}$ with $\alpha \in [\ell_e] \setminus  (J')^{e}_{i,j}$ obtaining a new partition $\{(P')^{e,\alpha}_{i,j} : \alpha \in [\ell_e] \setminus  (J')^{e}_{i,j}\}$ of $(V_i \times V_j) \setminus S'_{e,0}$ so that each $(P')^{\varepsilon, \alpha}_{i,j}$ still satisfies $\disc_2(\varepsilon_5)$ (or we could just hide it in a single set $Q^{e,\alpha}_{i,j}$ of maximal measure $\geq 1/\ell$ this time). We also let $(P')^{e,\alpha}_{i,j} := \emptyset$ for $\alpha \in  (J')^{e}_{i,j}$.
For each $1 \leq \alpha \leq b_e$, we let $(S')_{e, \alpha} := \bigcup_{(i,j) \notin \Gamma'_e} (P')^{\alpha}_{i,j}$. It follows that for all $1 \leq \alpha \leq b_e$ and all $(i,j)$, the density of $(S')_{e, \alpha}$ on $V_i \times V_j$ is either exactly $0$, or at least $F(\ell)/\ell$.

	Given any $\alpha, \beta, \gamma > 0$, we then have  $C := (S')_{\{1,2\},\alpha} \cap (S')_{\{1,3\},\beta} \cap (S')_{\{2,3\},\gamma}$ is the union of pairwise disjoint sets  $C_{i,j,k} := (P')^{\{1,2\},\alpha}_{i,j} \cap (P')^{\{1,3\},\beta}_{i,k} \cap (P')^{\{2,3\},\gamma}_{j,k}$ over all triples $(i,j,k)$ with each pair not in the corresponding $\Gamma'$. Using quasi-randomness of the $P'$'s, by the counting lemma for each such $(i,j,k)$  the set $C_{i,j,k}$ has density exactly $0$ or at least $(F(\ell) / \ell)^3 - \varepsilon_6$ on $V_i \times V_j \times V_k$, for $\varepsilon_6(\ell) \ll (F(\ell) / \ell)^3$. Hence, summing over all $(i,j,k)$,  $\mu(C)$ is at least $F'(\ell) \mu((S')_{\{1,2\},\alpha}) \mu((S')_{\{1,3\},\beta}) \mu((S')_{\{2,3\},\gamma})$ assuming $F$ was sufficiently small, and similarly to the proof of Proposition \ref{prop: str 2-stab reg implies disc} we still have that $E$ is homogeneous on $C$ outside of a set of measure $F'(\ell)|C|$ (this uses the modified definition of $(J')^{e}_{i,j}$ to get a fraction proportional to $F'(\ell)$ rather than $\varepsilon'$, taking advantage that $S_{e,0}$ was not involved in the repartitioning).
\end{proof}

\begin{remark}
It follows from Theorem \ref{thm:main} (combined with Proposition \ref{prop: str 2-stab reg implies disc} and \cite[Theorem 2.55]{terry2021irregular}) that the $(k+1)$-ary functions considered in Theorem \ref{thm:main} with $k+1$ and $d = k$, viewed as relations in continuous logic,  are NFOP$_k$ (and NOP$_2$ in the sense of Takeuchi \cite{takeuchi} for $k=3, d=2$). We note that these relations are  not slice-wise NIP in the sense of \cite{chernikov2020hypergraph}, so in particular not slice-wise stable (in the sense of \cite{chernikov2024perfect}) and not partition-wise NIP (in the sense of \cite{chernikov2021definable}). 

Indeed, consider random graphs on $XY$, $XZ$, $YZ$; let $\Omega$ consist of four points, ${a,b,c,d}$, of equal measure. We define $P_{x,y}$ to be $\{a,b\}$ if $(x,y)$ is an edge in the graph on $XY$ and to be $\{c,d\}$ otherwise; we define $Q_{x,z}$ to be $\{a,c\}$ if $(x,y)$ is in the graph on $XZ$ and $\{b,d\}$ otherwise; we define $R_{y,z}$ to be $\{a,d\}$ if $(y,z)$ is in the graph on $YZ$ and $\{b,c\}$ otherwise. If $(x,y,z)$ is a triple with all three pairs in the corresponding graph, $P_{x,y}\cap Q_{x,z}\cap R_{y,z}=\{a\}$; if exactly one edge is present, $P_{x,y}\cap Q_{x,z}\cap R_{y,z}$ is a singleton $\{b\}$, $\{c\}$, or $\{d\}$, depending on which edge. If $0$ or $2$ edges are present, $P_{x,y}\cap Q_{x,z}\cap R_{y,z}=\emptyset$. So the corresponding (continuous logic) formula $\mu(P_{x,y}\cap Q_{x,z}\cap R_{y,z})$ recovers the set of triples with an odd number of edges present. \end{remark}

\subsection{No single excluded substructure}\label{sec: half-simp and GS}

\begin{definition}\label{def: half-simp and GS}
\begin{enumerate}
	\item  The \emph{half-simplex} is the $3$-partite $3$-hypergraph where each part is a copy of $\mathbb{Q}\cap[0,1]$ and $(x,y,z)$ is an edge exactly when $x+y+z\geq 1$. (This defines the set of points outside of the standard simplex in $\mathbb{R}^{3}$. We could instead consider $x+y+z\leq 1$ to define the points inside the simplex --- this does not affect Theorem \ref{thm: embeds into GS3 and simplex}.)
	\item  For an ordinal $\kappa \in \omega \cup \{\omega\}$ and a fixed prime $p$, $\GS_p(\kappa)$ is the $3$-partite $3$-hypergraph where each part is a copy of $\mathbb{F}_{p}^{\kappa}$ and $(x,y,z)$ is a  hyper-edge exactly when, for the least $n < \kappa$ such that $x(n)+y(n)+z(n)\neq 0$ (calculated in $\mathbb{F}_p$) we have $x(n)+y(n)+z(n) = 1$ (and not an edge if there is no such $n$). We will simply write $\GS_p$ for $\GS_p(\omega)$. Note that for any $\kappa < \lambda \in \omega \cup \{\omega\}$, $\GS_p(\kappa)$ is isomorphic to an induced sub-hypergraph of $\GS_p(\lambda)$ (e.g.~via an embedding sending $x \in \mathbb{F}_p^{\kappa}$ to $x'$ with $x'(i) = x(i)$ for $i < \kappa$ and $x'(i) = 0$ for $i \geq \kappa$).
\end{enumerate}
\end{definition}
Terry and Wolf showed both of the families of finite induced sub-hypergraphs of the half-simplex (\cite[Theorem 2.55]{terry2021irregular}, referred to as HOP$_2$ there) and of $\GS_p$ for $p\geq 3$ (\cite[Theorem 2.58]{terry2021irregular}) fail to have  binary $\disc_{2,3}$ error, hence fail to have strong $2$-stability by Proposition \ref{prop: str 2-stab reg implies disc}.

Also, the half-simplex does not embed into $\GS_3(n)$ as an induced sub-hypergraph (as $\GS_3(n)$ does not have $4$-HOP$_2$ \cite[Appendix A.3]{terry2021higher}). Conversely, $\GS_3(2)$ does not embed into the half-simplex (as an induced sub-hypergraph):

\begin{definition}\label{def: monotone}
	We say that a hypergraph $\mathcal{H} = (E; X^1, X^2, X^3)$ is \emph{monotone} if there exist linear orders $<^u$ on $X^u$  for each $u \in [3]$ so that: if $x<^1 x' \in X^1$ and $y \in X^2, z \in X^3$ are arbitrary so that $(x,y,z)\in E$, then $(x',y,z)\in E$; and similarly for $y$ with respect to $<^2$ and for $z$ with respect to $<^3$.
\end{definition}

\begin{remark}
\begin{enumerate}
	\item The half-simplex is monotone, witnessed by each $<^u$ the standard linear order on $\mathbb{Q}$. 
	\item Every induced sub-hypergraph of a monotone hypergraph is monotone. In particular, every hypergraph that embeds into the half-simplex is monotone.
	\item  $\GS_3(2)$ is not monotone (as explained in the proof of Claim \ref{cla: one extra coordinate fixed}(3a), since each part $X^u$ contains sequences $\langle 00 \rangle , \langle 01 \rangle, \langle 02 \rangle$, so $\langle 0 \rangle^u$ splits in three directions for each $u \in [3]$, and $\langle 0 \rangle = I(\langle 0 \rangle, \langle 0 \rangle)$), hence $\GS_3(\kappa)$ does not embed into the half-simplex for any $\kappa \geq 2$.
\end{enumerate}
\end{remark}

 We now demonstrate  that any common induced sub-$3$-hypergraph of these two hypergraphs  \emph{does} have strong $2$-stability.  This shows, in particular, that strong $2$-stability cannot be characterized by a single family of excluded $3$-hypergraphs, partially addressing \cite[Conjecture 2.62]{terry2021irregular}.

\begin{theorem}\label{thm: embeds into GS3 and simplex}
Any sequence  $\mathcal{H}_i = (E_i; X^1_{i}, X^2_{i}, X^3_{i})$ for $i \in \mathbb{N}$ of finite 3-partite 3-hypergraphs which embed as induced $3$-partite sub-$3$-hypergraphs into both $\GS_3$ and into the half-simplex, satisfies perfect $2$-stable regularity.

\end{theorem}

\begin{proof}
	Let $\mathcal{H}_i = (E_i; X^1_{i}, X^2_{i}, X^3_{i})$ for $i \in \mathbb{N}$ be a sequence of finite 3-partite 3-hypergraphs which embed as induced $3$-partite sub-$3$-hypergraphs into both $\GS_3$ and into the half-simplex, and let $\mu^u_{i}$ be a probability measure on $X^u_{i}$ for $u \in [3]$. Identifying via the embeddings into $\GS_3$, we may assume $X^u_i \subseteq \mathbb{F}_3^{\mathbb{N}}$ and $E$ is the restriction of the edge relation in $\GS_3$ to $X^1_i\times X^2_i\times X^3_i$. And the embeddings  into the half-simplex induce from $(\mathbb{Q}, <)$ an ordering $<^u_i$ on each $X^u_i$ so that $E_i$ is \emph{monotone with respect to $<_i^u$} for each $u \in [3]$: if $x<_i^1 x' \in X^1_i$ and $(x,y,z)\in E_i$ then $(x',y,z)\in E_i$; and similarly for $y$ with respect to $<^2_i$ and for $z$ with respect to $<^3_i$.
	
	Let $\mathcal{U}$ be a non-principal ultrafilter on $\mathbb{N}$, and consider the infinite $3$-partite $3$-hypergraph $\mathcal{H} = (E; X^1, X^2, X^3) := \prod_{i \in \mathbb{N}} \mathcal{H}_i / \mathcal{U}$. We can identify $X^u$ with a subset of $ \mathbb{F}_3^{\mathbb{N}^{\ast}}$, where $\mathbb{N}^{\ast}$ is an ultrapower of $\mathbb{N}$, so that for any $(x^u \in X^{u} : u \in [3])$, either $x^1(n) + x^2(n) + x^3(n) =  0$ for all $n \in \mathbb{N}^{\ast}$, or there is a smallest such $n$ such that $x^1(n) + x^2(n) + x^3(n) >  0$; and $E(x^1,x^2, x^3)$ holds if and only if $x^1(n) + x^2(n) + x^3(n) = 1$ for the smallest such $n$. Note that $\mathbb{N}^{\ast}$ is not well-founded, but this holds by \L os' theorem.

		Formally, we can present $GS_3$ as a three-sorted first-order structure $\mathcal{M} = (Y, Z, \mathbb{F}_3; E, e)$ where $Y = 3^\mathbb{N}$ with no additional structure, $\mathbb{N}$ is equipped with the ordered ring structure, $\mathbb{F}_3$ with the field structure, function $e: (x, n) \in 3^{\mathbb{N}} \times \mathbb{N} \mapsto x(i) \in \mathbb{F}_3$ and $E \subseteq Y^3$ is definable via 
		\begin{gather*}
			F(x,y,z) \iff \\
			\exists n \in Z (\forall i <n, e(x,i) + e(y,i) + e(z,i) = 0 \land  e(x,n) + e(y,n) + e(z,n) = 1).
		\end{gather*}
		Then $\mathcal{H}_i$ is isomorphic to $(E^{\mathcal{M}} \upharpoonright X^1_i \times X^2_i \times X^3_i; X^1_i, X^2_i, X^3_i )$ for some $X^u_i \subset Y$, let $\mathcal{M}_i$ be the expansion of $\mathcal{M}$ by adding predicates for $X^u_i$.
		Taking the corresponding ultraproduct $\mathcal{M}^{\ast} := \prod_{i \in \mathbb{N}} \mathcal{M}_i / {\mathcal{U}}$, 		$\mathcal{M}^{\ast} = (Y^{\ast}, \mathbb{N}^{\ast}, \mathbb{F}_p; E^{\ast}, e^{\ast})$ we have that $\mathcal{H}$ is isomorphic to  $(E^{\ast} \upharpoonright X^1 \times X^2 \times X^3; X^1, X^2, X^3 )$ for internal sets $X^u = \prod_{i \in \mathbb{N}} X^u_i / \mathcal{U}$, and by \L os theorem, for every $x \in Y^{\ast}$, $e^{\ast}(x,-)$ is a function from $Z^{\ast} = \mathbb{N}^{\ast}$ to $\mathbb{F}_p$, and $E^{\ast}$ is defined in the same way as above. We will follow the notation for measures on ultraproducts from Section \ref{sec: ultraprod and perf reg}, so for $e \subseteq [3]$, $\mu^e$ is the ultralimit of the measures $\mu^e_i$ (and we will typically omit $e$ when it is clear from the context).

	Letting $<^{u} := \prod_{i \in \mathbb{N}} <^u_i / \mathcal{U}$, $<^u$ is a linear order on $X^u$ and $E$ is monotone with respect to $<^u$ (again, by \L os' theorem). Given two triples $(x^u : u \in [3]), (y^u : u \in [3])$ with $x^u, y^u \in X^u$, we will write $(x^1, x^2, x^3) \leq (y^1, y^2, y^3)$ if $x^u \leq^u y^u$ for all $u \in [3]$, and ``$<$'' when at least one of the inequalities is strict.

%
%
%
%
%
%
%
%
%

Given $n \in \mathbb{N}^{\ast}$, we write $\mathbb{F}_3^{n}$ for the set of functions $\{n' \in \mathbb{N}^{\ast} : 0 \leq n' < n\} \to \mathbb{F}_3$ (which we also think about as sequences of $\{0,1,2\}$). Given two sequences $\sigma,\tau \in \mathbb{F}_3^{ b}$ for $b \in \mathbb{N}^{\ast} \cup \{\mathbb{N}^{\ast}\}$, we define 
 $I(\sigma,\tau) \in \mathbb{F}_3^{ b}$ to be the sequence with $I(\sigma,\tau)(i)=-(\sigma(i)+\tau(i))$ for all $i$. In particular: $\sigma(i)+\tau(i)+I(\sigma,\tau)(i)=0$ for all $i$; and for any $\sigma, \tau, \upsilon$ we have $\upsilon = I(\sigma, \tau)$ if and only if $\sigma = I(\upsilon, \tau)$, if and only if $\tau = I(\sigma, \upsilon)$ (note that if $\{u,v,w\} = \{1,2,3\}$ and $x \in X^u, y \in X^v$, it is not necessary that $I(x,y) \in X^w$). 
For $b \in \mathbb{N}^{\ast}$, $|\sigma| := b$ denotes the (non-standard) length of $\sigma$. If $\sigma \in \mathbb{F}_3^{ b}, \tau \in \mathbb{F}_3^{ b'}$, we write $ \sigma \sqsubseteq \tau$ if $b \leq b'$ and $\sigma(i) = \tau(i)$ for all $i < b$ in $\mathbb{N}^{\ast}$ (and $ \sigma \sqsubset \tau$ if also $b < b'$). Given $\sigma \in \mathbb{F}_3^{ b}$, we write $[\sigma]$ for  the set $\{\tau \in \mathbb{F}_3^{\mathbb{N}^{\ast}} : \sigma \sqsubseteq \tau  \}$ of all sequences with prefix $\sigma$. If $\sigma \in \mathbb{F}_3^{ b}$ and $b' \leq b$, we write $\sigma \upharpoonright {b'}$ for the sequence $(\sigma(i) : i < b') \in \mathbb{F}_3^{ b'}$. We write $\sigma \perp \tau$ if neither $\sigma \sqsubseteq \tau$ nor $ \tau \sqsubseteq \sigma$.

We often write $\sigma^u$ to indicate that we are thinking of $\sigma$ as a prefix of an element of $X^u$. We say $\sigma^u$ is \emph{present} if $[\sigma^u]\cap X^u\neq\emptyset$.  For two sequences $\sigma \in \mathbb{F}_3^{b},  \tau \in \mathbb{F}_3^{b'} $ with $b,b' \in \mathbb{N}^{\ast} \cup \{\mathbb{N}^{\ast}\}$, we denote by $\sigma \cap \tau$ their maximal common initial segment. If both $\sigma, \tau$ are present, then $\sigma \cap \tau \in \mathbb{F}_3^{b''}$
 for some $b'' \in \mathbb{N}^{\ast}$ (by \L os' theorem).

We say $\sigma^u\in \mathbb{F}_3^{b}$ with $b \in \mathbb{N}^{\ast}$ is \emph{non-splitting} if $\sigma^u$ is present and there is a unique $j \in \mathbb{F}_{3}$ so that $[\sigma^u]\cap X^u\subseteq [\sigma^u{}^\frown\langle j\rangle]$. We call $j$ the \emph{unique extension} of $\sigma^u$. We say that $\sigma^u$ is a \emph{large split} if there are $j \neq j' \in \{0,1,2\}$ so that both $\mu([\sigma^u {}^\frown j]) >0$  and $\mu([\sigma^u {}^\frown j']) >0$. 

\begin{claim}\label{cla: ctbly many large splits}
There are at most countably many large splits.	
\end{claim}
\begin{claimproof}
	For each large split, consider the measure of the smallest of the
positive measure results of that split. If there were uncountably many large
splits, there would have to be some $\varepsilon > 0$ so that uncountably many of
these splits had that smallest result be at least $\varepsilon$. But there can only be finitely many of these (the subtree of these splits has both chains and antichains bounded in size by $1/\varepsilon$).
\end{claimproof}

\begin{definition}
	We define $Z^u$ to be those $x^u \in X^u$ for which there is some $v \in \{1,2,3\}$ (so $v$ could be either equal or not to $u$) and a large split $\sigma^v$ so that $\mu([x^u \upharpoonright |\sigma^v|]) > 0$ and $\mu([x^u \upharpoonright (|\sigma^v|+1)]) = 0$. Let $Z^{u,v} := (Z^u \times X^v) \cup (X^u \times Z^v)$.
\end{definition}
That is, $Z^u$ is those sequences where, faced with a split into two positive measure sets on a given level, they instead went a third, measure $0$ way. 
\begin{claim}\label{cla: Z meas 0}
$Z^{u,v} \in \mathcal{B}_{\{u,v\}}$ and $\mu(Z^{u,v}) = 0$.	
\end{claim}
\begin{proof}
	By Claim \ref{cla: ctbly many large splits}, the length $|\sigma^v|$ for $v \in  [3]$ and $\sigma^v$ a large split can take  at most countably many values in $\mathbb{N}^{\ast}$. Once $|\sigma^v|$ is fixed, there are at most countably many $\tau^u$ with $|\tau^u| = |\sigma^v|$ and $\mu([\tau^u]) > 0$ (as they are pairwise disjoint), and for each there are at most $3$ choices of $j \in \{0,1,2\}$ so that $[\tau^u {}^{\frown} j]$ has measure $0$. This shows that $Z^u$ is a union of countably many measure $0$ sets.
\end{proof}

The following is a crucial property of $E$ inherited from the embedding into the half-simplex:
\begin{claim}\label{cla: one extra coordinate fixed}
 Assume $n \in \mathbb{N}^{\ast}$, $\sigma_u \in \mathbb{F}_3^{n}$ are present (in the corresponding $X^u$) for $u \in [3]$ and $\sigma_3 = I(\sigma_1, \sigma_2)$. Then exactly one of the following happens:
 \begin{enumerate}
 	\item at least one of $\sigma_u$ is non-splitting (when  $I(\sigma_1,\sigma_2)$ is non-splitting, we write $j(\sigma_1,\sigma_2)$ for the unique extension of $I(\sigma_1,\sigma_2)$);
 	\item each $\sigma_u$ splits --- then:
\begin{enumerate}
	\item each $\sigma_u$ splits in exactly two directions, i.e.~there exist $j_1^u \neq j_2^u \in \{0,1,2 \}$ so that $[\sigma_u] \subseteq [\sigma_u {}^\frown j^u_1] \sqcup [\sigma_u {}^\frown j^u_2]$ and both $\sigma_u {}^\frown j^u_1, \sigma_u {}^\frown j^u_2$ are present in $X^u$;
	\item we have $x^u_1 <^u x^u_2$ for all $x^u_t \in [\sigma_u {}^\frown j^u_t], t \in [2]$;
	\item for any $(t_u : u \in [3]) \in [2]^3$, if $\sum_{u \in [3]} j^u_{t_u} = 0$ then there is some $(t'_u : u \in [3]) \in [2]^3$ so that
	\begin{itemize}
	\item either  $t'_u \leq  t_u$ for all $u$ and $\sum_{u \in [3]} j^u_{t_u} = 1$ (hence for all $x^u \in [\sigma_u {}^\frown j^u_{t_u}] \cap X^u$ we have $(x^1, x^2, x^3) \in E$  by (b) monotonicity of $E$),
	\item or $t_u \leq  t'_u$ for all $u$ and $\sum_{u \in [3]} j^u_{t_u} = 2$ (hence for all $x^u \in [\sigma_u {}^\frown j^u_{t_u}] \cap X^u$ we have $(x^1, x^2, x^3) \notin E$  by (b) monotonicity of $E$).
	\end{itemize}  
\end{enumerate}
	
	In this case we will write $j^w_t(\sigma^u_1, \sigma^v_2)$ for $j^w_t$ ($w \in [3], t \in [2]$).
 \end{enumerate}
 
  \end{claim}
  \begin{claimproof}
%
%
%

Assume each of $\sigma_u$ splits in $X^u$ for $u \in [3]$. Then there exist some $x^u_1 <^u x^u_2$ in $[\sigma_u] \cap X^u$ so that $x^u_1(n) \neq x^u_2(n)$ for all $u \in [3]$.

For (a), assume first that one of them, say $\sigma_3$, splits in all three directions. So we have in fact $x^3_1 <^3 x^3_2 <^3 x^3_3 \in [\sigma_3] \cap X^3$ with $\{x^3_t(n) : t \in [3] \} = \{0,1,2\}$. We also have $\{x^1_{t_1}(n) + x^2_{t_2}(n) : t_1, t_2 \in [2]  \} =  \{0,1,2\}$.  It follows that there exist some $t_1, t_2 \in [2]$ and  $t_3 < t'_3$ in $[3]$ so that $x^1_{t_1}(n) + x^2_{t_2}(n) + x^3_{t_3}(n) = 1$ and $x^1_{t_1}(n) + x^2_{t_2}(n) + x^3_{t'_3}(n) = 2$. Then, by definition of $E$, $(x^1_{t_1}, x^2_{t_2}, x^{3}_{t_3}) \in E$, $(x^1_{t_1}, x^2_{t_2}, x^{3}_{t'_3}) \notin E$, but also $x^{3}_{t_3} <^3 x^{3}_{t'_3}$ --- contradicting monotonicity of $E$.

For (b), assume that there exist also some $y_1 <^3 y_2$ in $[\sigma_3] \cap X^3$ so that $x^3_1(n) = y_2(n)$ and $x^3_2(n) = y_1(n)$. If the ordered pair $(x^3_1(n), x^3_2(n))$ is one of $(0,1), (1,2), (2,0)$, there is still some $b \in \{0,1,2\}$ so that $b + x^3_1(n) = 1, b+  x^3_2(n) =2$, so we get a contradiction to monotonicity of $E$ as above. But if $(x^3_1(n), x^3_2(n))$ is not one of these pairs, then necessarily $(y_1(n), y_2(n))$ is one of these pairs, so we can use $(y_1(n), y_2(n))$ in place of $(x^3_1(n), x^3_2(n))$  to contradict monotonicity.

For (c), assume that  $\bar{t}=(t_u : u \in [3])$ is such that $\sum_{u \in [3]} j^u_{t_u} = 0$. If at least two of the $t_u$'s are $=1$, say $\bar{t} = (1,1,2)$, we use that $\{j^1_{t_1} + j^2_{t_2} : t_1, t_2 \in [2]  \} =  \{0,1,2\}$ to find $t_1, t_2$ so that $j^1_{t_1} + j^2_{t_2} + 2 = 2$ (so automatically $(t_1, t_2, 2)$ is above $\bar{t}$).  Otherwise at least two of the $t_u$'s are $=2$, say $\bar{t} = (2,1,2)$, and we use that $\{j^1_{t_1} + j^3_{t_3} : t_1, t_3 \in [2]  \} =  \{0,1,2\}$ to find $t_1, t_3$ so that $j^1_{t_1} +  2 + j^3_{t_3} = 2$ (and automatically $(t_1, 1, t_3)$ is below $\bar{t}$).
  \end{claimproof}




 First we try to exhaust by pairwise disjoint rectangles of positive measure which decide $E$.

The first collection of sets we identify are those rectangles $[\sigma] \times [\tau]$ of positive measure where all three of $\sigma, \tau, I(\sigma,\tau)$ split simultaneously:
\begin{definition}\label{def: sR sets def}
 Given $u \neq v \in [3]$ and  $\sigma^u, \tau^v \in \mathbb{F}_3^{n}$ for $n \in \mathbb{N}^{\ast}$, suppose that 
\begin{enumerate}
\item $|\sigma^u|=|\tau^v| = n$,
\item $\mu([\sigma^u])>0$, $\mu([\tau^v])>0$, and $\mu([I(\sigma^u,\tau^v) \upharpoonright (n-1) ]  )>0$,
\item each of $ \sigma^u\upharpoonright (n-1), \tau^u\upharpoonright |(n-1), I(\sigma^u,\tau^v)\upharpoonright (n-1)$ splits.
\end{enumerate}
Then we set $sR^{u,v}_{\sigma,\tau} := [\sigma]\times[\tau]$ (``$sR$''  for ``splitting rectangle''; note that $sR^{u,v}_{\sigma,\tau} \in \mathcal{B}_{\{u,v\}}$ and $\mu(sR^{u,v}_{\sigma,\tau}) > 0$.)
\end{definition} 
Note that these are symmetric under permuting $u,v$.

The second collection of sets we identify are those where there is a single rectangle which simultaneously gets resolved into or out of $E$:
\begin{definition}\label{def: R sets def}
	For $u \neq v \in [3]$ and $\sigma^u, \tau^v \in \mathbb{F}_3^{n}, n \in \mathbb{N}^{\ast}$, suppose that
\begin{enumerate}
\item $|\sigma^u|=|\tau^v| = n$,
\item $\mu([\sigma^u])>0$, $\mu([\tau^v])>0$, and $\mu([I(\sigma^u,\tau^v) \upharpoonright (n-1) ]  )>0$,
\item $I(\sigma^u,\tau^v)\upharpoonright n-1$ is non-splitting,
\item $\sigma(n-1)+\tau(n-1)+j(\sigma^u\upharpoonright (n-1),\tau^v\upharpoonright (n-1))=j_0 > 0$.
\end{enumerate}
 Then we set $R^{u,v}_{\sigma,\tau} := [\sigma]\times[\tau]$. ($R$ stands for (non-splitting) ``rectangle''; note that $R^{u,v}_{\sigma,\tau} \in \mathcal{B}_{\{u,v\}}$ and  $\mu(R^{u,v}_{\sigma,\tau}) > 0$.)
\end{definition}

\begin{definition}
Fix $u < v \in [3]$. First we add $sR$ sets to $\mathcal{R}^{u,v}$ by induction on length $n \in \mathbb{N}^{\ast}$.  Given $n$, if the pair $(\sigma^u, \tau^v)$ with $\sigma,\tau \in \mathbb{F}_3^{n}$ satisfies Definition \ref{def: sR sets def} and $sR^{u,v}_{\sigma, \tau}$ is not contained in any of the previously selected sets $sR^{u,v}_{\sigma',\tau'} $ with $|\sigma'|=|\tau'| < n$, we add $sR^{u,v}_{\sigma, \tau}$ to $\mathcal{R}^{u,v}$.

 After this, we add to $\mathcal{R}^{u,v}$ all sets of the form $R^{u,v}_{\sigma,\tau}$  not contained in any of the previously added sets $sR^{u,v}_{\sigma',\tau'}$.
\end{definition}

\begin{claim}\label{cla: R sets disjoint and countable}
For fixed $u \neq v \in [3]$, all $R$-sets are pairwise disjoint, so there are at most countably many of them (as all of them have positive measure).

And all  sets in $\mathcal{R}^{u,v}$ are pairwise disjoint.  In particular there are at most countably many sets in $\mathcal{R}^{u,v}$ (as all of them have positive measure).

And every set of the form $R^{u,v}_{\sigma,\tau}$ or $sR^{u,v}_{\sigma,\tau}$ is contained in some set in $\mathcal{R}^{u,v}$.
\end{claim}
\begin{proof}

Note that for any $(\sigma_1, \tau_1), (\sigma_2, \tau_2)$ with $|\sigma_i| = |\tau_i|$ for both $i \in [2]$, if $([\sigma_1]\times[\tau_1])\cap([\sigma_2]\times[\tau_2]) \neq \emptyset$, then for some $i \in [2]$ we have both $\sigma_i \sqsubseteq \sigma_{3-i}$ and $\tau_i \sqsubseteq \tau_{3-i}$ (i.e.~one of the rectangles is contained in the other one). Hence, from the way they were chosen, all of the $sR$ sets in $\mathcal{R}$ are pairwise-disjoint, and every $R$ set in $\mathcal{R}$ is disjoint from every $sR$ set in $\mathcal{R}$. Suppose $(\sigma_1^u, \tau_1^v) \neq (\sigma_2^u, \tau_2^v)$	are as in Definition \ref{def: R sets def} and $([\sigma_1]\times[\tau_1])\cap([\sigma_2]\times[\tau_2])\neq\emptyset$. Then, without loss of generality, $\sigma_1\sqsubseteq\sigma_2$ and $\tau_1 \sqsubseteq \tau_2$, so $|\sigma_2| > |\sigma_1|$. But then $[I(\sigma_1,\tau_1)]\supseteq [I(\sigma_2,\tau_2)\upharpoonright (|\sigma_2|-1)]$. The latter set is non-empty by Definition \ref{def: R sets def}(2) for $(\sigma_2, \tau_2)$, while the former has to be empty by Definition \ref{def: R sets def}(3,4) for $(\sigma_1, \tau_1)$. This contradiction shows that $R^{u,v}_{\sigma_1, \tau_1} \cap R^{u,v}_{\sigma_2, \tau_2} = \emptyset$.
\end{proof}


We have related  measure $0$ sets ``adjacent'' to $R$-sets, and measure $0$ extensions determined by $sR$ sets:
\begin{definition}\label{def: RZ sets}
	Assume $\{u,v,w\} = \{1,2,3\}$.
	
	 For each $R^{u,v}_{\sigma,\tau}$ and any pair $\sigma',\tau'$ such that $|\sigma'|=|\tau'|=|\sigma|$, $\sigma'\upharpoonright (|\sigma|-1) =\sigma\upharpoonright (|\sigma|-1)$ and $\tau'\upharpoonright (|\sigma|-1)=\tau\upharpoonright (|\sigma|-1)$, if either $\mu([\sigma'])=0$ or $\mu([\tau'])=0$ (or both), 
we define a set $RZ^{u,v}_{\sigma',\tau'} := [\sigma'] \times [\tau']$ of measure $0$.

For each $sR^{u,v}_{\sigma, \tau} \in \mathcal{R}$ and any $(\sigma')^u, (\tau')^{v}, \rho^w$ with $|(\sigma')^u| = |(\tau')^{v}| = |\rho^w| = |\sigma|$ and $\sigma'\upharpoonright (|\sigma|-1) =\sigma\upharpoonright (|\sigma|-1)$, $\tau'\upharpoonright (|\sigma|-1)=\tau\upharpoonright (|\sigma|-1)$, $\rho \upharpoonright (|\sigma|-1)=I(\tau, \sigma) \upharpoonright (|\sigma|-1)$:
\begin{enumerate}
	\item if either $\mu([\sigma'])=0$ or $\mu([\tau'])=0$ (or both), we define $sRZ^{u,v}_{\sigma', \tau'} :=  [\sigma'] \times [\tau']$;
	\item if either $\mu([\sigma'])=0$ or $\mu([\rho])=0$ (or both), we define $sRZ^{u,w}_{\sigma', \rho} :=  [\sigma'] \times [\rho]$;
	\item if either $\mu([\tau'])=0$ or $\mu([\rho])=0$ (or both), we define $sRZ^{v,w}_{\tau', \rho} :=  [\tau'] \times [\rho]$.
\end{enumerate} 

We let $RZ^{u,v}$ be the (countable, by Claim \ref{cla: R sets disjoint and countable}) union of all $RZ^{u,v}_{\sigma',\tau'}$  and all $sRZ^{u,v}_{\sigma',\tau'}$ defined this way on $X^u \times X^v$, this set has measure $0$.
\end{definition} 
Before going on, we notice that cylinder intersection sets involving  $R$ or $sR$ sets  behave well:

\begin{claim}\label{cla: inters with two R sets}
Any cylinder intersection involving  two $R$ sets, or two $sR$ sets, or one $R$ set and one $sR$ set, regardless of the third set, is $E$-homogeneous.
\end{claim}\label{cla: inters of R sets are hom}
\begin{claimproof}
	Consider the (cylinder) intersection of two such sets, say $S^{1,2}_{\sigma,\tau}\cap T^{1,3}_{\sigma',\tau'}$, where each $S,T \in \{R,sR\}$.
	Assume it is non-empty, without loss of generality we may assume $\sigma \sqsubseteq \sigma'$ (as both sets are rectangles).  Let $\upsilon_0 := I(\sigma,\tau)\cap \tau'$. 
	 
	 First assume $I(\sigma, \tau) \perp \tau'$ (i.e.~$\upsilon_0 \sqsubset \tau'$ and $\upsilon_0 \sqsubset I(\sigma, \tau)$).  Then $j := \sigma(|\upsilon_0|)+\tau(|\upsilon_0|)+\tau'(|\upsilon_0|) \neq 0$. Then for any $(x,y,z) \in [\sigma] \times [\tau] \times [\tau'] \supseteq  S^{1,2}_{\sigma,\tau}\cap T^{1,3}_{\sigma',\tau'}$ we have $(x,y,z)\in E^j$ (where we write $E^1$ for $E$ and $E^2$ for the complement of $E$).
	 
	 The case $\tau' \sqsubset I(\sigma, \tau)$ is impossible as by assumption $|\tau'| = |\sigma'|  \geq |\sigma| = |I(\sigma, \tau)|$.
	 
	 In the remaining case $ I(\sigma, \tau) \sqsubseteq \tau'$. 
	 
	 If $S = R$ we get that the intersection is empty: by Definition \ref{def: R sets def}(3,4) for $R^{1,2}_{\sigma,\tau}$ we have $[I(\sigma,\tau)] \cap X^{3} = \emptyset$, so also $[\tau'] \cap X^{3} = \emptyset$. Otherwise $S = sR$, and by definition of $sR^{1,2}_{\sigma,\tau}$, each of $ \sigma \upharpoonright (|\sigma|-1), \tau \upharpoonright (|\sigma|-1), I(\sigma,\tau)\upharpoonright (|\sigma|-1)$ splits. By Claim \ref{cla: one extra coordinate fixed}(2c) it follows that there is some $j \neq 0$ so that all $(x,y,z)$ in $[\sigma] \times [\tau] \cap [I(\sigma,\tau)]$, hence in particular all $(x,y,z) \in sR^{1,2}_{\sigma,\tau}\cap T^{1,3}_{\sigma',\tau'}$,  are in $E^j$.
%
%
%
%
\end{claimproof}

We next observe the following general principle, which we will use repeatedly: 
\begin{claim}\label{cla: gen princ}
	Suppose that we have points $x^u \in X^u$ for $u \in [3]$ so that $x^3 \neq I(x^1,x^2)$ and, taking the least $i$ so that $x^1(i)+x^2(i)+x^3(i)>0$, we have $\mu([x^u\upharpoonright i+1])>0$ for all $u \in [3]$. Then:
	\begin{enumerate}
		\item either for at least one pair $u \neq v \in [3]$, $[x^u \upharpoonright i+1] \times [x^v\upharpoonright i+1]$ is an $R$ set (not necessarily in $\mathcal{R}^{u,v}$),
		\item or for every pair $u \neq v \in [3]$, $[x^u \upharpoonright i+1] \times [x^v\upharpoonright i+1]$ is an $sR$ set (not necessarily in $\mathcal{R}^{u,v}$).
	\end{enumerate} 
\end{claim}
\begin{proof}
By minimality of $i$, for every $u \in [3]$ we have $x^u \upharpoonright i = I(x^v \upharpoonright i, x^w \upharpoonright i )$ and $x^u(i) \neq I(x^v, x^w)(i)$ for $\{u,v,w\} = \{1,2,3\}$.

	Assume  first one of $x^u\upharpoonright i, u \in [3]$ is non-splitting. Say, without loss of generality, $x^1 \upharpoonright i$ is non-splitting. And $\mu([x^1\upharpoonright i]),  \mu([x^2 \upharpoonright i+1], \mu([x^3\upharpoonright i+1])>0$ by assumption. 
	Then $[x^2 \upharpoonright i+1] \times [x^3\upharpoonright i+1]$ satisfies Definition \ref{def: R sets def} of an $R$-set.

Otherwise all of $x^u\upharpoonright i, u \in [3]$ are splitting, and by minimality of $i$ and assumption as above we get that each of the sets $[x^u \upharpoonright i+1] \times [x^v\upharpoonright i+1]$ satisfies Definition \ref{def: sR sets def} of an $sR$-set.
\end{proof}

\begin{claim}\label{cla: R set and disj from R hom}
 Assume $\{1,2,3\} = \{u,v,w\}$. For any $S^{u,v}_{\sigma, \tau} \in \mathcal{R}^{u,v}$, where $S \in \{R, sR\}$, the cylinder intersection set 
	\begin{gather*}
	\{(x^1,x^2,x^3) \in X^1 \times X^2 \times X^3 : (x^u,x^v) \in S^{u,v}_{\sigma, \tau} \land \\
	 (x^u,x^w) \notin \bigcup\mathcal{R}^{u,w} \land  (x^v,x^w) \notin \bigcup\mathcal{R}^{v,w} \}	
	\end{gather*}
	is $0$-homogeneous for $E$.
\end{claim}
\begin{claimproof}
	Without loss of generality $u=1,v=2,w=3$.
	
	For any $(x,y,z)$ in this set, consider the least $i$ such that $x(i)+y(i)+z(i)> 0$ (or $\infty$).
	
	Suppose first $i<|\sigma|-1$ (in particular $z \upharpoonright i = I(\sigma, \tau) \upharpoonright i$).
	Then, by Definition \ref{def: R sets def}(2) if $S=R$ or by Definition \ref{def: sR sets def}(2) if $S = sR$,  $\mu([x\upharpoonright i+1]) \geq \mu([\sigma])>0$, $\mu([y\upharpoonright i+1]) \geq \mu([\tau])>0$, and $\mu([z\upharpoonright i+1]) \geq \mu([I(\sigma,\tau) \upharpoonright |
	\sigma| - 1 ]) >0$.   Note that $z\upharpoonright i$ splits, because $\mu([I(\sigma,\tau) \upharpoonright i+1]) \geq  \mu([I(\sigma,\tau) \upharpoonright |\sigma|-1]) >0$ and $z(i) \neq I(\sigma, \tau)(i)$ by assumption on $i$. Therefore, by Claim \ref{cla: gen princ}, at least one of $[x\upharpoonright i] \times [z\upharpoonright i]$ or $[y\upharpoonright i] \times [z\upharpoonright i]$ is either an $R$ set or an $sR$ set, and the corresponding pair $(x,z)$ or $(y,z)$ is in it, hence in some set in $\mathcal{R}$, contradicting the assumption.

%
	 
%
	
	So we must have $i \geq |\sigma|-1$. 
	
	If $S = R$, by definition of $R^{1,2}_{\sigma,\tau}$, $I(\sigma, \tau) \upharpoonright (|\sigma|-1)$ splits and $I(\sigma, \tau)$ is not present, so $x(|\sigma|-1) + y(|\sigma|-1) + z(|\sigma|-1) \neq 0$, hence  $i\leq |\sigma| -1 $ by minimality, so we must have $i = |\sigma|-1$.  It follows that, for any $(x,y,z)$ in our set, 	
	$(x,y,z)\in E^{j_0}$ (where $j_0 \in \{1,2\}$ is as in Definition \ref{def: R sets def} for $R^{1,2}_{\sigma, \tau}$).
	
	Otherwise $S = sR$.  By definition of $sR^{1,2}_{\sigma,\tau}$ and Claim \ref{cla: one extra coordinate fixed}(2c), each of $\sigma\upharpoonright (|\sigma|-1), \tau \upharpoonright (|\sigma|-1), I(\sigma,\tau) \upharpoonright (|\sigma|-1) $ splits in exactly two directions, say  $I(\sigma,\tau) \upharpoonright (|\sigma|-1)$ splits into $j_1 \neq j_2 \in \{0,1,2\}$.
	
	Assume $(x,y,z)$ in our set is such that $i > |\sigma| - 1$ (so $I(\sigma, \tau) \sqsubseteq z$ and, without loss of generality, $I(\sigma, \tau)(|\sigma|-1) =  j_1$). If $\mu([I(\sigma, \tau)]) > 0$, we get that $[\tau] \times [I(\sigma, \tau)]$ is an $sR^{2,3}$-set containing $(y,z)$, hence $(y,z)$ belongs to some set in $\mathcal{R}^{2,3}$ --- contradicting the assumption.
If  $\mu([I(\sigma, \tau)]) = 0$, we get that $(x,z) \in sRZ^{1,3}_{\sigma, I(\sigma, \tau)}$ (by case (2) of Definition \ref{def: RZ sets} for $sR^{1,2}_{\sigma, \tau}$; and also $(y,z) \in sRZ^{2,3}_{\tau, I(\sigma, \tau)}$ by case (3) of Definition \ref{def: RZ sets} for $sR^{1,2}_{\sigma, \tau}$).

	 Hence, in either case, all $(x,y,z)$ in our set with $(x,z)$ not in $sRZ^{1,3}$ (so outside of a measure $0$ set of triples) must satisfy $i = |\sigma| - 1$, hence $z(|\sigma| - 1) = j_2$. By Definition of $sR^{1,2}_{\sigma,\tau}$ and Claim \ref{cla: one extra coordinate fixed}(2c)  there is some $j \neq 0$ so that all $(x,y,z)$ in $[\sigma] \times [\tau] \cap [(I(\sigma,\tau) \upharpoonright |\sigma|-1)^{\frown} j_2]$, hence all $(x,y,z)$ in our set minus a measure $0$ set,   are in $E^j$.
	\end{claimproof}

Putting  Claims \ref{cla: inters of R sets are hom} and \ref{cla: R set and disj from R hom} together, any cylinder intersection set in which one set is in $\mathcal{R}$ and the others are either in $\mathcal{R}$ or disjoint from all sets in $\mathcal{R}$, is $0$-homogeneous. It remains to partition the complement of the  union of $\mathcal{R}$ sets so that the remaining cylinder intersection sets are homogeneous.


\begin{claim}\label{cla: defining i(x)}
Every $x \in X^u$ satisfies exactly one of the following.
\begin{enumerate}
\item We have $\mu([x \upharpoonright i]) >0$ for all $i \in \mathbb{N}^{\ast}$.  Then $\mu(\{x\}) > 0$. In this case we say that $i(x) = \infty$.
	\item There exists least $i \in \mathbb{N}^{\ast}$ so that $\mu([x \upharpoonright i])  = 0$ (then $i \geq 1$ and $\mu([x \upharpoonright i-1]) > 0$). We denote this $i$ by $i(x)$ and say that \emph{$i(x)$ is a successor}.
	\item There exists a Dedekind cut $(L,U)$ of $(\mathbb{N}^{\ast}, <)$ so that $U$ has no least element, $\mu([x \upharpoonright i]) > 0$ for all $i \in L$, $\mu([x \upharpoonright i]) = 0$ for all $i \in U$, and for every $\varepsilon > 0$ there is some $i \in L$ with  $0<\mu([x \upharpoonright i])<\epsilon$.   In this case we let $i(x)$ be the element of the Dedekind completion of $(\mathbb{N}^{\ast}, <)$ realizing this cut, and say that \emph{$i(x)$ is a limit}.

	Given $x,y$ with at least one of $i(x), i(y)$ a limit, we evaluate $i(x) < i(y)$  in the Dedekind completion.
\end{enumerate}	
And sets of the form $\{x \in X^u : i(x) \textrm{ is a successor}\}$, $\{x \in X^u : i(x) \textrm{ is a limit}\}$ are in $\mathcal{B}_{\{u\}}$, and $\{(x,y) \in X^u \times X^v : i(x) < i(y)\} \in \mathcal{B}_{\{u,v\}} $.

\end{claim}
\begin{proof}
Let $x \in X^u$ be arbitrary, say $x = (x_i : i \in \mathbb{N}) / \mathcal{U}$ with $x_i \in X^u_i$.

First assume that $\mu([x \upharpoonright n]) > 0$ for all $n \in \mathbb{N}^{\ast}$. 

\noindent Then there exists $0 < \varepsilon \in \mathbb{R}$ so that $\mu([x \upharpoonright n]) \geq \varepsilon$ for all $n \in \mathbb{N}^{\ast}$. Indeed, otherwise, for every $\varepsilon \in \mathbb{Q}_{>0}$ there exists some $n_{\varepsilon} \in \mathbb{N}^{\ast}$ so that $\mathcal{M}^{\ast} \models m_{y} < \varepsilon. \varphi(y; x, n_{\varepsilon})$, where $\varphi(y; \alpha, \beta) = \forall z \in Z ( z < \beta  \rightarrow e(\alpha,z) = e(y,z) ) \in \mathcal{L}$. Then, by $\aleph_1$-saturation of the ultraproduct $\mathcal{M}^{\ast}$, we can find some $n \in \mathbb{N}^{\ast}$ so that $\mathcal{M}^{\ast} \models m_{y} < \varepsilon. \varphi(y; x, n)$ for all $\varepsilon \in \mathbb{Q}_{>0}$ simultaneously, so $\mu([x \upharpoonright n]) = 0$. 

\noindent Then for a $\mathcal{U}$-large set of $i \in \mathbb{N}$ we have: for all $n \in \mathbb{N}$, $\mu_i([x_i \upharpoonright n]) \geq \varepsilon/2$. But as $x_i \in X^u_i$ and $X^u_i$ is finite, we can chose $n_i \in \mathbb{N}$ so that $\{x_i\} = [x_i \upharpoonright n] \cap X^u_i$, so $\mu_i(\{x_i\}) \geq \varepsilon/2$.  Hence also $\mu(\{x\}) \geq \varepsilon/2$, so case (1) holds.

Otherwise we must have $\mu( [x \upharpoonright n]) = 0$ for some $n \in \mathbb{N}^{\ast}$, and the set of such $n$ is obviously upwards closed. If for every $\varepsilon >0$ there is some $n$ with $0 < \mu( [x \upharpoonright n]) < \varepsilon$ we are in case (3). 

Assume not. That is, there is some $\varepsilon^{\ast} >0$ so that for every $n \in \mathbb{N}^{\ast}$ we have $\mu( [x \upharpoonright n]) > 0 \Rightarrow \mu( [x \upharpoonright n]) \geq \varepsilon^{\ast}$. Then the set $\{n  \in \mathbb{N}^{\ast} : \mu( [x \upharpoonright n]) > 0 \}$ is definable (with parameters) in $\mathcal{M}^{\ast}$ via $\neg (m_{y} < \varepsilon^{\ast}/2). \varphi(y; x, n)$. But by \L os' theorem, every \emph{definable} subset of $\mathbb{N}^{\ast}$ has a minimal element, so we are in case (2).
\end{proof}

There could be a lot of rectangles of measure $0$ remaining which add up to a positive measure set. We will group them into finitely many groups so that all rectangles in each group behave in the same way with respect to $E$, and add the union of each group to the partition.

\begin{definition}
	First, when $(x,y)\in X^u\times X^v$ and $x,y$ are both \emph{atoms} (i.e.~$\mu(\{x\}), \mu(\{y\}) > 0$) and $(x,y)$ is not contained in any set in $\mathcal{R}^{u,v}$, we define $A_{x,y}^{u,v} := \{(x,y)\}$. Clearly there are only countably many such sets (as they are of positive measure).
\end{definition}




We  add sets where the first event to happen for $x,y,I(x,y)$ is non-splitting.

\begin{definition}\label{def: C leq sets}
	For $u \neq v \in [3]$ and $j_0 \in \{1,2\}$, let $C^{u \leq v}_{j_0}$ be  the set of all $(x,y) \in X^u \times X^v$ such that:
\begin{enumerate}
\item $(x,y)$ is not in any  set in $\mathcal{R}$,
\item $i(x) \leq i(y)$,
\item $i(x)$ is a successor,
\item there is a unique $j' \in \{0,1,2\}$ with $\mu([(x\upharpoonright i(x)-1)^{\frown} j'])>0$,
\item there is a unique $j'' \in \{0,1,2\}$ with $\mu([(y\upharpoonright i(x)-1)^{\frown} j''])>0$ (we could have $j'' = y(i(x)-1)$ or not),
\item $I(x,y)\upharpoonright (i(x)-1)$ is non-splitting (with the unique extension $j(x,y)$),
  \item \begin{itemize}
  	\item either:  ($u<v$ or $i(x) < i(y)$) and $x(i(x)-1) + j'' + j(x,y)=j_0$, 
  	\item or: $u > v$ and $i(x) = i(y)$ and $x(i(x)-1) + j'' + j(x,y)=j_0$ and $j' + y(i(x)-1) + j(x,y) = 0$.
  \end{itemize}
  
 \end{enumerate}

\end{definition}

\begin{definition}
We define $C^{u < v}_{j_0}$ as in Definition \ref{def: C leq sets} except we require:
\begin{itemize}
	\item  instead of (2): $i(x)<i(y)$, and
	\item  instead of (6) and (7):
	\begin{itemize}
	\item either $ I(x,y)\upharpoonright (i(x)-1)$ is not present and $j_0 = 1$;
	\item  or  $I(x,y)\upharpoonright (i(x)-1)$ is non-splitting with the unique extension $j(x,y)$, and  $x(i(x)-1)+j''+j=0$ and $j'+y(i(x)-1)+j=j_0$.
	\end{itemize}
\end{itemize}

\end{definition}

\begin{definition}
We define $C^{u = v}_{j_0}$ as in Definition \ref{def: C leq sets} except we require:
\begin{itemize}
	\item instead of (2): $i(x)=i(y) < \infty$, and
	\item  instead of (6) and (7):
	\begin{itemize}
	\item either $ I(x,y)\upharpoonright (i(x)-1)$ is not present and $j_0 = 1$;
	\item  or  $I(x,y)\upharpoonright (i(x)-1)$ is non-splitting with the unique extension $j(x,y)$, and $x(i(x)-1)+j''+j=0$ and  $j'+ y(i(x)-1) +j=0$  and $x(i(x)-1)+y(i(x)-1)+j=j_0$.
	\end{itemize}
\end{itemize}

\end{definition}

We also add sets where the first event to happen for $x,y,I(x,y)$ is splitting.

\begin{definition}\label{def: sR< sets}
	Given $u \neq v \in [3]$ and $j^u_1, j^u_2,  j^v, j^w_1, j^w_2 \in \mathbb{F}_3$ (where $w$ is such that $\{u,v,w\} = \{1,2,3\}$) and $s^u \in [2]$ so that $j^u_1 \neq j^u_2, j^w_1 \neq j^w_2$  we define $sC^{u < v}_{j^u_1, j^u_2, s^u, j^v, j^w_1, j^w_2}$ as the set of all $(x,y) \in X^u \times X^v$ such that:
\begin{enumerate}
\item $(x,y)$ is not in any set in $\mathcal{R}$,
\item $i(x) < i(y)$,
\item $i(x)$ is a successor,
\item there is a unique $j' \in \{0,1,2\}$ with $\mu([(x\upharpoonright i(x)-1)^{\frown} j'])>0$,
%
\item $y(i(x)-1) = j^v$,
\item 
\begin{itemize}
	\item if $s^u = 1$: $x(i(x)-1) = j^u_1$, $j' = j^u_2$ and there \emph{exists} $x' \in [(x\upharpoonright i(x)-1)^{\frown} j'] \cap X^{u}$ with $x <^u x'$,
	\item if $s^u = 2$:  $j' = j^u_1$, $x(i(x)-1) = j^u_2$ and \emph{for all}  $x' \in [(x\upharpoonright i(x)-1)^{\frown} j'] \cap X^{u}$ we have $x >^{u} x'$,
\end{itemize}
\item  $I(x,y) \upharpoonright (i(x)-1)$ splits, and $(j^w_1, j^w_2)$ is the least pair in the lexicographic order on $\{0,1,2\} \times \{0,1,2\}$ so that $j^w_1 \neq j^w_2$ and \emph{there exist}  $z_1 , z_2$  so that $z_1 <^w z_2$ and $z_t \in [(I(x,y)\upharpoonright i(x)-1)^{\frown} j^w_t ] \cap X^{w}$.
\end{enumerate}
\end{definition}

\begin{definition}\label{def: sR= sets}
	Given $u \neq v \in [3]$ and $j^u_1, j^u_2,  j^v_1, j^v_2, j^w_1, j^w_2 \in \mathbb{F}_3$ (where $w$ is such that $\{u,v,w\} = \{1,2,3\}$) and $s^u, s^v \in [2]$ so that $j^u_1 \neq j^u_2, j^v_1 \neq j^v_2, j^w_1 \neq j^w_2$  we define $sC^{u = v}_{j^u_1, j^u_2, s^u, j^v_1, j^v_2, s^v, j^w_1, j^w_2}$ as the set of all $(x,y) \in X^u \times X^v$ such that:
\begin{enumerate}
\item $(x,y)$ is not in any set in $\mathcal{R}$,
\item $i(x) = i(y)$ is a successor,
\item ($s^u = 1$ and $x(i(x)-1) = j^u_1$) or ($s^u = 2$ and $x(i(x)-1) = j^u_2$),
\item ($s^v = 1$ and $y(i(x)-1) = j^v_1$) or ($s^v = 2$ and $y(i(x)-1) = j^v_2$),
\item $x\upharpoonright (i(x)-1)$ splits in exactly two directions $j^u_1 \neq j^u_2$, $y \upharpoonright (i(x)-1)$ splits in exactly two directions $j^v_1 \neq j^v_2$, and $I(x,y) \upharpoonright (i(x)-1)$ splits in exactly two directions $j^w_1\neq j^w_2$,
\item  $x_1 <^u x_2$ for all $x_t \in [(x\upharpoonright i(x)-1)^{\frown} j^u_t] \cap X^u$, $y_1 <^v y_2$ for all $y_t \in [(y \upharpoonright i(x)-1)^{\frown} j^v_t] \cap X^v$,  $z_1 <^w z_2$ for all $z_t \in [(I(x,y)\upharpoonright i(x)-1)^{\frown} j^w_t ] \cap X^{w}$.
\end{enumerate}
\end{definition}


\begin{definition}
	Finally, $L^{u<v}\subseteq X^u\times X^v$ is the set of $(x,y)$ so that $i(x)<i(y)$ and $i(x)$ is a limit and not contained in any of the sets above, and $L^{u=v}$ is the set of $(x,y)$ so that $i(x)=i(y)$ is a limit and not contained in any of the sets above.
\end{definition}

\begin{definition}
	 For fixed $u < v \in [3]$, we add to $\mathcal{P}^{u,v}$ all sets in $\mathcal{R}^{u,v}$, all sets of the form $A_{x,y}^{u,v}$, all sets of the form $C^{u < v}$ \emph{and} $C^{v < u}$, all sets of the form   $C^{u \leq v}_{j_0}$ \emph{and} $C^{v \leq u}_{j_0}$ (we split into two cases in Definition \ref{def: C leq sets}(7) to prevent  $C^{u \leq v}_{j_0}$ from overlapping with any of the $C^{v \leq u}_{j'_0}$ on points with $i(x)=i(y)$), all sets $C^{u = v}_{j_0}$, all sets of the form $sC^{u<v}$ \emph{and} $sC^{v<u}$, and all sets of the form $L^{u<v}$ and $L^{v<u}$, and all sets $L^{u=v}$.
Note that all these sets are pairwise disjoint.
\end{definition}

\begin{lemma}
 For every $u<v \in [3]$, $\mathcal{P}^{u,v}$  is a partition of $X^u \times X^v$, up to a measure $0$ set.
\end{lemma}
\begin{proof}
  Consider some $(x,y) \in X^u \times X^v$ not contained in any set in $\mathcal{R}^{u,v}$ or $Z^{u,v}$, or $RZ^{u,v}$.
  
  \

\noindent 1)  Assume $i(x)=i(y) < \infty$ is a successor. 

Then $x\upharpoonright (i(x)-1)$ splits. Indeed, we have $\mu([x \upharpoonright i(x)]) = 0$ and $\mu(x\upharpoonright (i(x)-1)) > 0$ by definition of $i(x)$, so $\mu([(x\upharpoonright (i(x)-1)) {}^\frown j']) > 0$ for some $j' \neq j^x := x(i(x)-1)$. Hence $(x\upharpoonright (i(x)-1)) {}^\frown j'$ is present, and $x \upharpoonright i(x)$ is present as $x \in [x \upharpoonright i(x)] \cap X^u$. Since $(x,y) \notin Z^{u,v}$, $x\upharpoonright i(x)-1$ is not a large split, so $j'$ is unique with $\mu([(x\upharpoonright i(x)-1) {}^{\frown} j'])>0$.

Similarly, $y\upharpoonright (i(y)-1)$ splits and there is a unique $j'' \neq j^y := y(i(y)-1)$ with $\mu([(y\upharpoonright i(y)-1) {}^{\frown} j''])>0$.

\

\noindent 1.1) Assume  $I(x,y)\upharpoonright (i(x)-1)$ is non-splitting, with unique extension $j := j(x,y)$.

If $j^x + j'' + j = 0$ and $j' +  j^y + j = 0$, then necessarily $x(i(y)-1) + y(i(y)-1) + j \neq 0$ (using $j' \neq x(i(y)-1) $), so $(x,y) \in C^{u = v}_{j^x + j^y + j}$.

If $j^x  + j'' + j \neq 0$ and $u < v$, we get $(x,y) \in C^{u \leq v}_{j^x  + j'' + j}$ (by the first case of Definition \ref{def: C leq sets}(7)). 

If $j' +  j^y + j \neq 0$ and $u > v$, reversing the roles of the coordinates we get $(x,y) \in C^{v \leq u}_{ j^y + j' + j}$ (by the first case of Definition \ref{def: C leq sets}(7)).

If  $j' +  j^y + j = 0$, $u < v$ and $j^x  + j'' + j \neq 0$, we get $(x,y) \in C^{u \leq v}_{j^x +  j'' + j}$ (falling into the second case of  Definition \ref{def: C leq sets}(7)).

Finally, if  $j^x  + j'' + j = 0$, $u < v$ and $j' +  j^y + j \neq 0$, reversing the roles of the coordinates we get $(x,y) \in C^{v \leq u}_{j^y +  j' + j}$ (falling into the second case of  Definition \ref{def: C leq sets}(7)).

%
%
%
%
%

\

\noindent 1.2)  Assume $I(x,y)\upharpoonright (i(x)-1)$ splits. 

Then, by Claim \ref{cla: one extra coordinate fixed}:
\begin{itemize}
	\item  $x\upharpoonright (i(x)-1)$ splits in exactly two directions $j' \neq j^x$, and there is $\square^u \in \{<,>\}$ so that $x_1 \square^u x_2$ for all $x_1 \in [(x\upharpoonright i(x)-1)^{\frown} j^x], x_2 \in [(x\upharpoonright i(x)-1)^{\frown} j']$; if $\square^u$ is ``$<$'' we let  $s^u := 1, j^u_1 := j^x, j^u_2 := j'$, and if $\square^u$ is ``$>$'' we let $s^u := 2, j^u_1 := j', j^u_2 := j^x$;
	\item  $y \upharpoonright (i(x)-1)$ splits in exactly two directions $j'' \neq j^y$, and there is $\square^v \in \{<,>\}$ so that $y_1 \square^v y_2$ for all $y_1 \in [(y \upharpoonright i(x)-1)^{\frown} j^y], y_2 \in [(y \upharpoonright i(x)-1)^{\frown} j'']$; if $\square^v$ is ``$<$'' we let  $s^v := 1, j^v_1 := j^y, j^v_2 := j''$, and if $\square^u$ is ``$>$'' we let $s^v := 2, j^v_1 := j'', j^v_2 := j^y$; 
	\item  $I(x,y)\upharpoonright (i(x)-1)$ splits in exactly two directions $j^w_1 \neq j^w_2$ so that $z_1 <^w z_2$ for all $z_t \in [(I(x,y)\upharpoonright i(x)-1)^{\frown} j^w_t ] \cap X^{w}$.
\end{itemize}
 Then $(x,y)$ is contained in $sC^{u = v}_{j^u_1, j^u_2, s^u, j^v_1, j^v_2, s^v, j^w_1, j^w_2}$.
 
 \

\noindent 1.3) Assume $I(x,y)\upharpoonright (i(x)-1)$  is not present.
 
 Then $(x,y) \in C^{u=v}_{1}$.

 \

\noindent 2) Assume $i(x)=i(y)$ is a limit.

 If $i(x) = i(y) < \infty$ then $(x,y)$ is contained in $L^{u=v}$, and if $i(x)=i(y)=\infty$, they $(x,y)$ is contained in $A_{\{x,y\}}$ by Claim \ref{cla: defining i(x)}.

\

\noindent 3) Otherwise, one of $i(x), i(y)$ is strictly smaller; without loss of generality, assume $i(x)<i(y)$.

\

\noindent 3.1) Assume $i(x)$ is a successor. 

As $(x,y) \notin Z^{u,v}$, both $x \upharpoonright (i(x)-1)$ and $y \upharpoonright (i(x)-1)$ are not large splits. Let $j' \neq x(i(x)-1)$ be unique with $\mu([(x \upharpoonright i(x)-1)^{\frown} j']) > 0$. 

As $i(x) < i(y)$, we must have $\mu([y \upharpoonright (i(x)-1)]) > 0$ (so $y(i(x)-1) = j''$).

\

\noindent 3.1.1) Assume $I(x,y)\upharpoonright (i(x)-1)$ does not split, with unique extension $j$.
If $x(i(x)-1)+j''+j \neq 0$ we have $(x,y) \in C^{1 \leq 2}_{x(i(x)-1)+j''+j}$.

 Otherwise $x(i(x)-1)+j''+j = 0$ and $j'+y(i(x)-1)+j \neq 0$, so $(x,y) \in C^{1<2}_{x(i(x)-1)+y(i(x)-1)+j}$.
 
 \

\noindent 3.1.2) Assume $I(x,y)\upharpoonright (i(x)-1)$ splits.

It could split in three directions. Nevertheless, there exists at least one pair $j^w_1 \neq j^w_2$  so that $z_1 <^w z_2$ and $z_t \in [(I(x,y)\upharpoonright i(x)-1)^{\frown} j^w_t ] \cap X^{w}$, and we pick the least such pair $(j^w_1, j^w_2)$ lexicographically. 

If there \emph{exists} $x' \in [(x\upharpoonright i(x)-1)^{\frown} j'] \cap X^{u}$ with $x <^u x'$, we set $j^u_1 := x(i(x)-1)$, $j^u_2 := j'$ and $s^u := 1$.

Otherwise for all $x' \in [(x\upharpoonright i(x)-1)^{\frown} j'] \cap X^{u}$ we have $x >^u x'$, then we set $j^u_2 := x(i(x)-1)$, $j^u_1 := j'$ and $s^u := 2$.

Then 
 $(x,y)$ is in 
 $sC^{u < v}_{j^u_1, j^u_2, s^u, y(i(x)-1), j^w_1, j^w_2}$.

\

\noindent 3.1.3) Assume $I(x,y)\upharpoonright (i(x)-1)$ is not present.

Then $(x,y) \in C^{1<2}_1$.

\

\noindent  3.2) If $i(x)$ is a limit and $<\infty$, and $(x,y)$ is not in any of the sets above, we get that  $(x,y)$ is either in $L^{u<v}$ or in $L^{u=v}$. Otherwise $i(x) = i(y) = \infty$ by minimality of $i(x)$, and we have $(x,y) \in A^{u,v}_{x,y}$.
\end{proof}

\begin{theorem}
  Any positive measure cylinder intersection of three sets in $\mathcal{P}^{1,2}, \mathcal{P}^{1,3}, \mathcal{P}^{2,3}$ respectively is $0$-homogeneous for $E$.
\end{theorem}
\begin{proof}
We consider a cylinder intersection set $S = S^{1,2} \cap S^{1,3} \cap S^{2,3}$ with $S^{u,v} \in \mathcal{P}^{u,v}$.

  We have already considered the case where at least one of the three intersecting sets is in $\mathcal{R}$ (Claims \ref{cla: inters with two R sets} and \ref{cla: R set and disj from R hom}). 
  
  Next, we consider cases involving $A$ sets. If two sets are $A$ sets, the third must be as well (since the intersection is non-empty), and the cylinder intersection has at most one element and is certainly $E$-homogeneous.

  Observe that any combination of the remaining sets uniquely determines the order of $(i(x),i(y),i(z))$ (in the same way for all $(x,y,z)$ in the cylinder intersection $S$). So let us assume that $i(x)$ is minimal (possibly equal to other values). The sets in the cylinder intersection also determine whether $i(x)$ is a successor, a limit, or $\infty$.

  \

\noindent 1) \emph{Assume first $i(x)$ is a successor}, so the two sets $S^{1,2} \cap S^{1,3}$ containing the coordinate for $X^1$ must be $C$ or $sC$  sets. 
  
    Fix $(x,y,z)$ in the intersection $S$, and assume there is an $i<i(x)-1$ with $x(i)+y(i)+z(i)>0$, fix a minimal such $i$. Then $\mu([x \upharpoonright i+1] ) \geq \mu([x\upharpoonright i(x)-1])>0$ (by definition of $i(x)$) and $\mu([y \upharpoonright i+1] ) \geq \mu([y\upharpoonright i(x)-1])>0,  \mu([z  \upharpoonright i+1] ) \geq \mu([z\upharpoonright i(x)-1])>0$ (by definition of $i(y), i(z)$ using that $i(x) \leq i(y), i(z)$ by assumption). By Claim \ref{cla: gen princ}, one of the pairs $(x,y), (x,z), (y,z)$ would be contained in either an $R$ set or in an $sR$ set, hence in some set in $\mathcal{R}$. Since this cannot happen, 
    \begin{itemize}
    	\item every $(x,y,z) \in S$ satisfies $z\upharpoonright (i(x)-1) = I(x ,y)\upharpoonright (i(x)-1)$.
    \end{itemize} 
  
\noindent 1.1) Assume at least one of the sets $S^{1,2}, S^{1,3} $ is a $C$ set, say $S^{1,2}$ is a $C$ set.   Then $I(x,y) \upharpoonright  (i(x) - 1)$ does not split (for $S^{1,2}$ either a $C^{1<2}$ set or a $C^{1=2}$ set ---  it has to be present since $z$ is in $[I(x,y) \upharpoonright  (i(x) - 1)]$ by the above), with the unique extension $j(x,y)$, so we must have 
 \begin{itemize}
 	\item $j(x,y) = z(i(x)-1)$. 
 \end{itemize}

\ 
  
\noindent 1.1.1)   Assume $S^{1,2} = C^{1 \leq 2}_{j_0}$ for $j_0 \in \{1,2\}$ first. 
  
  Fix $x$, and consider the set of all $y$ so that $(x,y) \in C^{1 \leq 2}_{j_0}$ but $j''(x,y) \neq y(i(x)-1) $. 
  Each such $y$ belongs to a set of measure $0$ of the form $[\sigma {}^\frown j^{\ast}]$ for some $\sigma \in \mathbb{F}_{3}^{i(x)-1}$ with $\mu([\sigma]) > 0$  (namely, $\sigma := y \upharpoonright (i(x) -1)$) and some $j^{\ast} \in \{0,1,2\}$. As all such $[\sigma]$ of fixed length are pairwise disjoint, there can be at most countably choices for $\sigma$. Hence the set of all such $y$ has measure $0$. Then, by Fubini,  all pairs $(x,y) \in C^{1 \leq 2}_{j_0}$ outside of a measure $0$ set satisfy $ j''(x,y) = y(i(x)-1)$. Hence, by definition of $C^{1 \leq 2}_{j_0}$,  for all $(x,y,z) \in S$ outside of a measure $0$ set we must have $x(i(x)-1)+y(i(x)-1)+j=j_0$, so $x(i(x)-1)+y(i(x)-1)+z(i(x)-1)=j_0$, so $(x,y,z) \in E^{j_0}$.

  \

\noindent 1.1.2) Assume $S^{1,2} = C^{1<2}_{j_0}$. 

Then $i(x) < i(y)$, so we must have $j'' = y(i(x)-1)$, $x(i(x)-1)+y(i(x)-1)+j=0$ and $j'+y(i(x)-1)+j=j_0$. Then all $(x,y,z) \in S$ satisfy:
\begin{itemize}
	\item $j = I(x,y)(i(x)-1)$,
	\item $\mu([(x\upharpoonright i(x)-1)^{\frown} j'])>0$ for a unique $j' \neq x( i(x)-1)$,
	\item $\mu([(y\upharpoonright i(x)]) >0$ as $i(x) < i(y)$,
	\item $\mu([(z\upharpoonright i(x)]) >0$ (as $I(x,y) \upharpoonright  (i(x) - 1) = z\upharpoonright (i(x)-1)$ does not split and has positive measure as $i(x) \leq i(z)$, hence $z (i(x)-1)$ is the unique extension, necessarily of positive measure).
\end{itemize}
It follows that $(x,y)$ is in $RZ^{1,2}_{(x\upharpoonright i(x)-1)^{\frown} j', y \upharpoonright i(x) }$ (adjacent to $R^{1,2}_{(x\upharpoonright i(x)-1)^{\frown} j', y\upharpoonright i(x)}$). This shows that $S$ is empty in this case.


\

\noindent 1.1.3) Assume $S^{1,2} = C^{1=2}_{j_0}$. 

By the above we have $z\upharpoonright (i(x)-1) = I(x ,y)\upharpoonright (i(x)-1)$, and by definition of $C^{1=2}_{j_0}$ we have $I(x,y) \upharpoonright  (i(x) - 1)$ does not split, hence $z (i(x)-1) = j$ is the unique extension.

So for every $(x,y,z)$ in $S$ we have $x(i(x)-1) + y(i(x)-1) + z(i(x)-1) = x(i(x)-1) + y(i(x)-1) + j = j_0$, hence $S$ is contained in $E^{j_0}$.

\

\noindent 1.2) So we assume now both $S^{1,2}$ and $S^{1,3}$ are $sC$ sets. Say $S^{1,2}$ is either $sC^{1 = 2}_{j^1_1, j^1_2, s^1, j^2_1, j^2_2, s^2, j^3_1, j^3_2}$ or $sC^{1 < 2}_{j^1_1, j^1_2, s^1, j^2, j^3_1, j^3_2}$.   And $S^{1,3}$ is either $sC^{1 = 3}_{k^1_1, k^1_2, r^1, k^3_1, k^3_2, r^3, k^2_1, k^2_2}$ or $sC^{1 < 3}_{k^1_1, k^1_2, r^1, k^3, k^2_1, k^2_2}$.

 Assume $(x,y,z) \in S$ is arbitrary. We have $z\upharpoonright (i(x)-1) = I(x ,y)\upharpoonright (i(x)-1)$ (by the above), $x \upharpoonright (i(x)-1)$ splits (as $(x,y)$ is in the $sC^{1=2}$ or $sC^{1<2}$ set above),  $y \upharpoonright (i(x)-1)$ splits (either because $(x,y)$ is in the $sC^{1=2}$ set above, or because $(x,z)$ is in the $sC^{1=3}$ or $sC^{1<3}$ set above), and $z \upharpoonright (i(x)-1)$ splits (as $(x,y)$ is in the $sC^{1=2}$ or $sC^{1<2}$ set above). Then, using Claim \ref{cla: gen princ}(2) and Definitions \ref{def: sR< sets} and \ref{def: sR< sets}, we have: 
   
   \begin{itemize}
   	\item  $x \upharpoonright (i(x)-1)$ splits in exactly two directions $j^{1}_1  \neq  j^{1}_2$, and $x(i(x)-1) = j^1_{s^1}$ (as $(x,y)$ is in the $sC^{1=2}$ or $sC^{1 <2}$ set above),
   	\item  $y \upharpoonright (i(x)-1)$ splits in exactly two directions $j^{2}_1 \neq  j^{2}_2$ and  $y(i(x)-1) = j^2_{s^2}$ (given by $(x,y) \in sC^{1=2}$; and if not using that $(x,z)$ is in $sC^{1=3}$ or $sC^{1<3}$ and taking $j^2_1 := k^2_1, j^2_2 := k^2_2$ and, as $(x,y) \in sC^{1 \leq 2}$, we take $s^2$ so that $j^2 = k^2_{s^2}$,
   	\item  $z \upharpoonright (i(x)-1)$ splits in exactly two directions $j^{3}_1 \neq j^{3}_2$ (as $(x,y)$ is in the $sC^{1=2}$ or $sC^{1 <2}$ set above) and $z(i(x)-1) = j^3_{s^3}$ (where if $(x,z)$ is in the $sC^{1=3}$ set above we take $s^3 := r^3$; and if  $(x,z)$ is in the $sC^{1<3}$ set above we take  $s^3$ so that $k^3 = j^3_{s^3}$),
   	\item  $x_1 <^u x_2$ for all $x_t \in [(x\upharpoonright i(x)-1)^{\frown} j^u_t] \cap X^u$, $y_1 <^v y_2$ for all $y_t \in [(y \upharpoonright i(x)-1)^{\frown} j^v_t] \cap X^v$,  $z_1 <^w z_2$ for all $z_t \in [(I(x,y)\upharpoonright i(x)-1)^{\frown} j^w_t ] \cap X^{w}$. In case that some of our sets were $sC^{<}$ sets, we use that, by  Claim \ref{cla: gen princ}(2), when all three prefixes split, each of this ``for all'' is equivalent to ``for some'', which was used in Definition \ref{def: sR< sets}.
   \end{itemize}

   As in Claim \ref{cla: one extra coordinate fixed}(2c),  there exists some $(t^1, t^2, t^3) \in [2]$ \emph{independent of $(x,y,z)$} so that: 
     	\begin{itemize}
	\item either  $s^\alpha \leq t^{\alpha}$ for all $\alpha \in [3]$  and $\sum_{\alpha \in [3]} j^{\alpha}_{t_{\alpha}} = 1$: then, by monotonicity of $E$, $S \subseteq E^1$;
	\item or $s^\alpha \geq t^{\alpha}$ for all $\alpha \in [3]$  and $\sum_{\alpha \in [3]} j^{\alpha}_{t_{\alpha}} = 2$: then, by monotonicity of $E$, $S \subseteq E^2$.
	\end{itemize}

   \

%

 \noindent  2) Next, suppose we are in a cylinder intersection set $S$ where \emph{$i(x)$ is a limit}. 
  
  We claim this cylinder intersection set has measure $0$. Fix $(y,z) \in X^v \times X^w$, and consider the slice $S_{y,z}$ consisting of those $x$ with $(x,y,z)$ in the cylinder intersection set $S$. By Claim \ref{cla: defining i(x)}, for any $\varepsilon>0$ and any $x\in S_{y,z}$, there is some $\sigma$ with $x\in[\sigma]$ and $0<\mu([\sigma])<\varepsilon$; therefore we can cover $S_{y,z}$ with sets of the form $[\sigma]$ with $0<\mu([\sigma])<\varepsilon$.

  We claim that $S_{y,z}$ must be contained in a single such set (then, by Fubini, $\mu(S) \leq \varepsilon $ for all $\varepsilon >0$, hence equal to $0$). Indeed, assume we have  $\sigma \neq \sigma'$ so that $\mu([\sigma]), \mu([\sigma']) > 0$ and $x \in [\sigma] \cap S_{y,z}, x' \in [\sigma'] \cap S_{y,z}$.   Then $\sigma, \sigma'$ have some maximum common initial segment $\sigma_0=\sigma\cap\sigma'$, with $\sigma_0 \in \mathbb{F}_3^{n}$ for some $n \in \mathbb{N}^{\ast}$ (as both $\sigma^u, (\sigma')^u$ are present). 
  So there are some $j\neq j' \in \{0,1,2\}$ with $\sigma_0 {}^{\frown} j\sqsubseteq \sigma$ and $\sigma_0 {}^{\frown} j'\sqsubseteq \sigma'$. 
  
Then at least one of  $j+ y(|\sigma_0|)+z(|\sigma_0|)\neq 0$ or $j'+ y(|\sigma_0|)+z(|\sigma_0|)\neq 0$, say the former one (the latter case is symmetric).

As $x \in [\sigma]$, we have $j = x(|\sigma_0|)$ and $\mu(x \upharpoonright |\sigma_0|+1) >0$. As $i(x) \leq i(y), i(z)$ we have $\mu([y \upharpoonright |\sigma_0|+1]) >0 $, $\mu([z \upharpoonright |\sigma_0|+1]) >0$.
  
  By Claim \ref{cla: gen princ}, this implies that at least one of $(x,y), (x,z), (y,z)$ is in some set in $\mathcal{R}$, contrary to the assumption.

  \

\noindent 3)  In the final case, $i(x)=\infty$. In this case, since $i(x)$ is minimal, $i(x)=i(y)=i(z)=\infty$, and so all three cylinder sets are $A$ sets by Claim \ref{cla: defining i(x)}, and therefore the cylinder intersection set is a singleton (or empty), so certainly $E$-homogeneous.
\end{proof}

\end{proof} 

\section{Acknowledgements}

We would like to thank  Karim Adiprasito, Vitaly Bergelson, Ehud Hrushovski, Paolo Marimon, Anand Pillay, Sergei Starchenko, Terence Tao and Tamar Ziegler for helpful conversations around the topic of the paper. Chernikov was partially supported by the NSF Research Grant DMS-2246598. Towsner was partially supported by the NSF Research Grant DMS-2054379.

\bibliographystyle{alpha}
\bibliography{refs}

\begin{thebibliography}{BYBHU08}

\bibitem[AFN07]{alon2007efficient}
Noga Alon, Eldar Fischer, and Ilan Newman.
\newblock Efficient testing of bipartite graphs for forbidden induced
  subgraphs.
\newblock {\em SIAM Journal on Computing}, 37(3):959--976, 2007.

\bibitem[BY13]{yaacov2013theories}
Ita{\"\i} Ben~Yaacov.
\newblock On theories of random variables.
\newblock {\em Israel Journal of Mathematics}, 194(2):957--1012, 2013.

\bibitem[BYBHU08]{yaacov2008model}
Ita{\"\i} Ben~Yaacov, Alexander Berenstein, C~Ward Henson, and Alexander
  Usvyatsov.
\newblock Model theory for metric structures.
\newblock {\em London Mathematical Society Lecture Note Series}, 350:315, 2008.

\bibitem[CCP24]{chavarria2024continuous}
Nicolas Chavarria, Gabriel Conant, and Anand Pillay.
\newblock Continuous stable regularity.
\newblock {\em Journal of the London Mathematical Society}, 109(1):e12822,
  2024.

\bibitem[Che23]{CheOber}
Artem Chernikov.
\newblock Towards higher classification theory.
\newblock In {\em Model Theory: Combinatorics, Groups, Valued Fields and
  Neostability. {Abstracts} from the workshop held {January} 8--14, 2023},
  volume~20 of {\em Oberwolfach Workshop Reports}, pages 129--134.
  Mathematisches Forschungsinstitut Oberwolfach, 2023.

\bibitem[CPT19]{chernikov2019n}
Artem Chernikov, Daniel Palacin, and Kota Takeuchi.
\newblock On n-dependence.
\newblock {\em Notre Dame Journal of Formal Logic}, 60(2), 2019.

\bibitem[CS21]{chernikov2021definable}
Artem Chernikov and Sergei Starchenko.
\newblock Definable regularity lemmas for {NIP} hypergraphs.
\newblock {\em The Quarterly Journal of Mathematics}, 72(4):1401--1433, 2021.

\bibitem[CT20]{chernikov2020hypergraph}
Artem Chernikov and Henry Towsner.
\newblock Hypergraph regularity and higher arity {VC}-dimension.
\newblock {\em Preprint, arXiv:2010.00726}, 2020.

\bibitem[CT24a]{arXiv:2406.18772}
Artem Chernikov and Henry Towsner.
\newblock Intersecting sets in probability spaces and {Shelah}'s
  classification.
\newblock Preprint, {arXiv}:2406.18772, 2024.

\bibitem[CT24b]{chernikov2024perfect}
Artem Chernikov and Henry Towsner.
\newblock Perfect stable regularity lemma and slice-wise stable hypergraphs.
\newblock {\em Preprint, arXiv:2402.07870}, 2024.

\bibitem[FR02]{frankl2002extremal}
Peter Frankl and Vojt{\v{e}}ch R{\"o}dl.
\newblock Extremal problems on set systems.
\newblock {\em Random Structures \& Algorithms}, 20(2):131--164, 2002.

\bibitem[Gow07]{gowers2007hypergraph}
W~Timothy Gowers.
\newblock Hypergraph regularity and the multidimensional {S}zemer{\'e}di
  theorem.
\newblock {\em Annals of Mathematics}, pages 897--946, 2007.

\bibitem[Gro52]{grothendieck1952criteres}
Alexandre Grothendieck.
\newblock Crit{\`e}res de compacit{\'e} dans les espaces fonctionnels
  g{\'e}n{\'e}raux.
\newblock {\em American Journal of Mathematics}, pages 168--186, 1952.

\bibitem[Hru12]{hrushovski2012stable}
Ehud Hrushovski.
\newblock Stable group theory and approximate subgroups.
\newblock {\em Journal of the American Mathematical Society}, 25(1):189--243,
  2012.

\bibitem[Hru24]{hrushovski2024approximate}
Ehud Hrushovski.
\newblock Approximate equivalence relations.
\newblock {\em Model Theory}, 3(2):317--416, 2024.

\bibitem[KM81]{krivine1981espaces}
Jean-Louis Krivine and Bernard Maurey.
\newblock Espaces de {B}anach stables.
\newblock {\em Israel Journal of Mathematics}, 39(4):273--295, 1981.

\bibitem[MS14]{malliaris2014regularity}
Maryanthe Malliaris and Saharon Shelah.
\newblock Regularity lemmas for stable graphs.
\newblock {\em Transactions of the American Mathematical Society},
  366(3):1551--1585, 2014.

\bibitem[NRS06]{nagle2006counting}
Brendan Nagle, Vojt{\v{e}}ch R{\"o}dl, and Mathias Schacht.
\newblock The counting lemma for regular k-uniform hypergraphs.
\newblock {\em Random Structures \& Algorithms}, 28(2):113--179, 2006.

\bibitem[PS13]{arXiv:1310.7538}
Anand Pillay and Sergei Starchenko.
\newblock Remarks on {Tao}'s algebraic regularity lemma.
\newblock Preprint, {arXiv}:1310.7538, 2013.

\bibitem[RS04]{rodl2004regularity}
Vojt{\v{e}}ch R{\"o}dl and Jozef Skokan.
\newblock Regularity lemma for k-uniform hypergraphs.
\newblock {\em Random Structures \& Algorithms}, 25(1):1--42, 2004.

\bibitem[Tak17]{takeuchi}
Kota Takeuchi.
\newblock On 2-order property.
\newblock {\em Slides from a talk given at the Asian Logic Conference 2017,
  Daejeon, Korea}, 2017.

\bibitem[Tao06]{tao2006variant}
Terence Tao.
\newblock A variant of the hypergraph removal lemma.
\newblock {\em Journal of combinatorial theory, Series A}, 113(7):1257--1280,
  2006.

\bibitem[Tao13]{Tao}
Terrence Tao.
\newblock ``{A} spectral theory proof of the algebraic regularity lemma'',
  blogpost.
\newblock
  \url{https://terrytao.wordpress.com/2013/10/29/a-spectral-theory-proof-of-the-algebraic-regularity-lemma/},
  2013.

\bibitem[Tao15]{zbMATH06476712}
Terence Tao.
\newblock Expanding polynomials over finite fields of large characteristic, and
  a regularity lemma for definable sets.
\newblock {\em Contrib. Discrete Math.}, 10(1):22--98, 2015.

\bibitem[TW21a]{terry2021higher}
Caroline Terry and Julia Wolf.
\newblock Higher-order generalizations of stability and arithmetic regularity.
\newblock {\em Preprint, arXiv:2111.01739}, 2021.

\bibitem[TW21b]{terry2021irregular}
Caroline Terry and Julia Wolf.
\newblock Irregular triads in 3-uniform hypergraphs.
\newblock {\em Memoirs of the American Mathematical Society, accepted
  (arXiv:2111.01737)}, 2021.

\end{thebibliography}

\end{document}